\documentclass[11pt]{amsart}
\usepackage{amsmath}
\usepackage{subfig}
\usepackage{float}
\usepackage{graphicx}
\RequirePackage{natbib}
\usepackage{color}
\usepackage{hyperref}
\allowdisplaybreaks
\newtheorem{lemma}{Lemma}[section]
\newtheorem{theorem}[lemma]{Theorem}

\newtheorem{proposition}[lemma]{Proposition}

\newtheorem{corollary}[lemma]{Corollary}

\newtheorem{claim*}{Claim}

\newtheorem*{problem*}{Problem}

\theoremstyle{definition}

\theoremstyle{remark}

\newtheorem{remark}[lemma]{Remark}

\newtheorem*{question*}{Question}

\numberwithin{equation}{section}
\numberwithin{table}{section}
\newcommand{\I}{\mathcal{I}}

\newcommand{ \R}{ \mathbb R }
\newcommand{\E}{ \mathbb E}
\newcommand{\N}{ \mathbb N}
\renewcommand{\Pr}{ \mathbb P}

\newcommand {\la}{\langle}
\newcommand{\ra}{\rangle}

\DeclareSymbolFont{bbold}{U}{bbold}{m}{n}
\DeclareSymbolFontAlphabet{\mathbbold}{bbold}
\newcommand{\ind}{\mathbbold{1}}

\begin{document}

\title[Evolution of patch-selection in stochastic environments]{Protected polymorphisms and evolutionary stability of patch-selection strategies in stochastic environments }

\author[S.N. Evans]{Steven N. Evans}
\thanks{S.N.E. was supported in part by NSF grant DMS-0907639 and NIH grant 1R01GM109454-01}
\address{Department of Statistics \#3860\\
 367 Evans Hall \\
 University of California \\
  Berkeley, CA  94720-3860 \\
   USA}
\email{evans@stat.berkeley.edu}

\author[A. Hening]{Alexandru Hening }
\thanks{A.H. was supported by EPSRC grant EP/K034316/1}
\address{Department of Statistics \\
 1 South Parks Road \\
 Oxford OX1 3TG \\
 United Kingdom}
 \email{hening@stats.ox.ac.uk}

\author[S.J. Schreiber]{Sebastian J. Schreiber}
\thanks{S.J.S. was supported in part by  NSF grants EF-0928987 and
DMS-1022639}
\address{Department of Evolution and Ecology\\
 University of California\\
   Davis, CA 956116 \\
    USA}
\email{sschreiber@ucdavis.edu}


\begin{abstract}

We consider a population living in a patchy environment
that varies stochastically in space and time.  The population is composed
of two morphs (that is, individuals of the same species with
different genotypes). In terms of survival and reproductive
success, the associated phenotypes differ only in
their habitat selection strategies.
We compute invasion rates corresponding to the rates at which the abundance
of an initially rare morph increases in the
presence of the other morph
established at equilibrium.  If both
morphs have positive invasion rates when rare, then there is an equilibrium
distribution such that the two morphs coexist; that is, there is
a protected polymorphism for habitat selection.
Alternatively, if one morph has a negative invasion rate
when rare, then it is asymptotically displaced by the other morph
under all initial conditions where both morphs are present.
We refine the characterization of an
evolutionary stable strategy for habitat selection
from [Schreiber, 2012] in a mathematically
rigorous manner.  We provide a necessary and sufficient condition
for the existence of an ESS that uses all patches and determine when
using a single patch is an ESS.  We also provide an explicit formula
for the ESS when there are two habitat types.  We show that adding
environmental stochasticity results in an ESS that, when compared to
the ESS for the corresponding model without stochasticity, spends
less time in patches with larger carrying capacities and possibly makes
use of sink patches, thereby practicing a spatial form of bet hedging.
\end{abstract}

\maketitle

\section{Introduction}
Habitat selection by individuals impacts key attributes of a population including its spatial distribution,  temporal fluctuations in its abundance, and its genetic composition. In environmentally heterogeneous landscapes, individuals selecting more favorable habitats are more likely to survive or reproduce. As population densities increase in these habitats, individuals may benefit by selecting previously unused habitats. Thus, both environmental conditions and density-dependent feedbacks generate selective pressures on habitat selection. Under equilibrium conditions, spatial heterogeneity can select for populations exhibiting an ideal-free distribution--equal per-capita growth rates in all occupied patches and lower per-capita growth rates if individuals moved into unoccupied patches~\citep{fretwell-lucas-69}. Under non-equilibrium conditions, spatial-temporal heterogeneity can select for individuals occupying sink habitats in which the per-capita growth rate is always negative~\citep{holt-97,jansen-yoshimura-98}. Environmental heterogeneity can also promote coexistence of genotypes only differing in their habitat choices~\citep{jaenike-holt-91}. Despite significant advances in the mathematical theory for habitat selection under equilibrium conditions, a mathematical theory for habitat selection in stochastic environments is largely lacking. Here, we take a step to addressing this mathematical shortfall while at the same  gaining new insights into the evolution of habitat selection for populations living in stochastic, patchy environments.

Since the classic paper \citet{fretwell-lucas-69}, the ideal-free distribution has been studied extensively from empirical, theoretical, and mathematical perspectives.  Empirical support for ideal-free distributions exists for many taxa including fish~\citep{godin-keenleyside-84,oksanen-etal-95,haugen-etal-06},  birds~\citep{harper-82,doncaster-etal-97}, mammals~\citep{beckmann-berger-03}, and insects~\citep{dreisig-95}. For example, \citet{oksanen-etal-95} found that armored catfish in Panamanian stream pools were distributed such that the resource availability per catfish was equal in all occupied pools, despite significant variation in light availability across these occupied pools. Theoreticians have identified several ``non-ideal'' mechanisms (e.g. sedentarism, adaptive movement with finite speed, density-dependent dispersal) that, under equilibrium conditions, generate an ideal-free distribution~\citep{hastings-83,cosner-05,gejji-etal-12}. For example, at equilibrium, sedentary populations achieve an ideal-free distribution provided, paradoxically, the populations initially occupied all habitat patches. While many early studies asserted that the ideal free distribution is an evolutionarily stable strategy (ESS)~\citep{fretwell-lucas-69,vanbaalen-sabelis-93,amnat-00}, only recent advanced nonlinear analyses fully verified this assertion~\citep{cressman-etal-04,cressman-krivan-06,cressman-krivan-10,cantrell-etal-07,cantrell-etal-10,cantrell-etal-12}.

In nature, observed habitat occupancies are frequently less extreme than predicted by the ideal-free distribution: individuals underuse higher quality habitats and overuse lower quality habitats compared to theoretical predictions~\citep{milinski-79,tregenza-95}. Notably, populations occupying sink habitats have been documented in many species~\citep{sokurenko-etal-06,tittler-etal-06,robinson-etal-08,anderson-geber-10}. One possible explanation for these observations is that populations experience temporal as well as spatial variation in environmental conditions and, consequently, theory based on equilibrium assumptions tells an incomplete story. In support of this explanation, several theoretical studies have shown that occupation of sink habitats should evolve when temporal variation is sufficiently great in other habitats~\citep{holt-97,jansen-yoshimura-98,holt-barfield-01,amnat-12}. These theoretical developments, however, rely on linearizations of density-dependent models, and do not analyze the dynamics of competing genotypes, the ultimate basis for evolutionary change due to natural selection. Hence, these studies leave unanswered the question, ``Does the linear analysis correctly identify competitive exclusion in pairwise interactions that is the basis for the analysis of evolutionarily stable strategies?''

Within populations, individuals can exhibit different habitat selection strategies, and there is some evidence these differences can be genetically based~\citep{via-90,jaenike-holt-91}. For instance, some individuals of the fruit fly species \emph{Drosophila tripunctata}  prefer tomato host plants (one potential habitat for its larvae) while others  prefer mushrooms (another potential habitat), and these differences are based on two genetically independent traits, settling behavior and ovipositor site preference~\citep{jaenike-85}.  \citet{jaenike-holt-91} found that genetic variation in habitat selection is common, especially in arthropods and mollusks. Furthermore, they demonstrated using mathematical models that this genetic variation can stem from density-dependent regulation occurring locally within each habitat. Specifically, Jaenike and Holt write ``frequency-dependent selection favors alleles that confer upon their carriers a preference for underused habitats, even if there is no genetic variation in how well individuals are adapted to the different habitat'' \citep[p.S78]{jaenike-holt-91}. Their analysis, however, doesn't account for temporal fluctuations in environmental conditions and this raises the question, ``Does environmental stochasticity facilitate or hinder the maintenance of genetic variation in habitat selection?''

To answer the aforementioned questions, we provide an in-depth analysis of a model introduced in \citep{amnat-12}. The single genotype (i.e. monomorphic) version of this model and a characterization of its dynamics are given in Section 2. The competing genotype (i.e. dimorphic) version of the model and invasion rates of each genotype when rare are  introduced in Section 3. In Section 4, we prove that these invasion rates determine the long-term fate of each of the genotypes. Specifically, if both genotypes have positive invasion rates when rare, then there is a positive stationary distribution under which the genotypes coexist. Alternatively, if one genotype has a negative invasion rate when rare, then it is asymptotically displaced by the other genotype. These result allows us to use the invasion rates when rare to explore conditions supporting a protected polymorphism for habitat selection. In Section 5, we refine the characterization of an evolutionary stable strategy for habitat selection from \citep{amnat-12} in a mathematically rigorous manner, and provide an explicit formula for this ESS when there are two habitat types. Section 6 concludes with a discussion of how our results relate to the existing literature and identifies future challenges for the theory of habitat selection in stochastic environments.

\section{The Monomorphic Model}

To set the stage for two competing populations spread over several patches, we start with a single population living in one patch. Let $Z_t$ be the population abundance at time $t \ge 0$. The stochastic process $(Z_t)_{t \ge 0}$ is governed by the It\^o stochastic logistic equation
\begin{equation}\label{e:Z}
d Z_t = Z_t (\mu  - \kappa Z_t) \, dt + \sigma Z_t \, dW_t,
\end{equation}
where $\mu$ is the intrinsic rate of growth of the population in the absence of stochasticity, $\kappa$ is the strength of intraspecific competition, $\sigma^2>0$ is the infinitesimal variance parameter of the stochastic growth rate, and $(W_t)_{t\geq 0}$ is a standard Brownian motion. The process $(Z_t)_{t\geq 0}$ is a strong Markov process with continuous paths. We  call an object with such properties a \textit{diffusion}.

As shown in our first proposition,  the process $(Z_t)_{t\geq 0}$ lives in the positive half line $\R_{++} :=(0,\infty)$; that is, if we start it in a strictly positive state, then it never hits zero. Furthermore, the long-term behavior of the process is determined by the stochastic rate of growth $\mu - \frac{\sigma^2}{2}$. When the stochastic growth rate is negative the population abundance converges asymptotically to zero with probability one.  On the other hand,  when this parameter is positive the distribution of the abundance converges to an equilibrium given by a Gamma distribution.  These results are well-known, but,  as introduction to the methods used to prove our main results, we provide a proof in Appendix A.

\begin{proposition}\label{P:Z} Consider the diffusion process
$(Z_t)_{t\geq 0}$ given by the stochastic differential equation
\eqref{e:Z}.
\begin{itemize}
\item
The stochastic differential equation has a unique strong solution
that is defined for all $t \ge 0$ and is given by
\[
Z_t = \frac{Z_0 \exp((\mu-\sigma^2/2)t+\sigma W_t)}{1+Z_0 \frac{\mu}{\kappa}\int_0^t \exp((\mu-\sigma^2/2)s+\sigma W_s)ds}.
\]
\item
If $Z_0 = z > 0$, then $Z_t > 0$ for all $t \ge 0$ almost surely.
\item
If $\mu - \frac{\sigma^2}{2}<0$, then $\lim_{t\to\infty} Z_t=0$ almost surely.
\item
If $\mu - \frac{\sigma^2}{2}=0$, then $\liminf_{t\to\infty} Z_t=0$ almost surely, $\limsup_{t\to\infty} Z_t=\infty$ almost surely, and
$\lim_{t\rightarrow \infty}\frac{1}{t}\int_0^t Z_s\,ds = 0$ almost surely.
\item
If $ \mu - \frac{\sigma^2}{2} >0$, then $(Z_t)_{t\geq 0}$ has a unique stationary distribution $\rho$ on $\R_{++}$ with Gamma density
$g(x)= \frac{1}{\Gamma(k)\theta^k}x^{k-1}e^{-\frac{x}{\theta}}$,
where
\[
\theta := \frac{\sigma^2}{2  \kappa}
\text{ and }
k := \frac{2\mu}{\sigma^2}-1.
\]
Moreover, if $Z_0 = z > 0$, then
\[
\lim_{t\to\infty} \frac{1}{t}\int_0^t h(Z_s) \, ds = \int_0^\infty h(x)g(x) \, dx \quad \text{almost surely}
\]
for any Borel function $h:\R_{++}\to \R$ with $\int_0^\infty |h(x)|g(x) \, dx<\infty$.
In particular,
\begin{equation*}
\lim_{t\rightarrow \infty}\frac{1}{t}\int_0^t Z_s\,ds =
\frac{1}{\kappa}\cdot \left(\mu - \frac{\sigma^2}{2}\right) \quad \text{almost surely}.
\end{equation*}
\end{itemize}
\end{proposition}

Next, we consider a population living in a spatially heterogeneous environment with $n$ different patches. These patches may represent distinct habitats, patches of the same habitat type, or  combinations thereof. The abundance of the population in the $i$-th patch at time $t \ge 0$ is $\bar X_t^i$.
Let $(\bar X_t^i)_{t \ge 0}$ be given by
\begin{equation}
\label{eq:base}
d\bar X_t^i = \bar X_t^i\left(\mu_i    -\kappa_i \bar X_t^i \right) \, dt + \bar X_t^i \, dE_t^i,
\end{equation}
where $\mu_i$ is the intrinsic rate of growth the population in patch $i$
in the absence of stochasticity, $\kappa_i$ is the strength of
intraspecific competition in patch $i$, and $E^i_t=\sum_j \gamma_{ji} B^j_t$
for a standard multivariate Brownian motion $(B^1,\dots,B^n)^T$
on $\R^n$ and an $n \times n$ matrix $\Gamma :=(\gamma_{ij})$.
The infinitesimal covariance matrix for the non-standard
Brownian motion
$(E_t^1,\dots, E_t^n)$ is  $\Sigma=(\sigma_{ij}):=\Gamma^T \Gamma$.

The populations in the various patches described by equation \eqref{eq:base}
are coupled only by the spatial correlations present in the driving Brownian motion $(E_t^1,\dots, E_t^n)$.
We further couple the population dynamics across patches by assuming the fraction of population in patch $i$ equals $\alpha_i$ for all time. This spatial distribution can be realized at the scale of the individual when, as described in greater detail in Remark~\ref{R:prelimiting}, individuals disperse rapidly and independently of one another in such a manner  that the fraction of time spent in patch $i$  equals  $\alpha_i$ for each individual. Under this assumption, we call $\alpha=(\alpha_1,\alpha_2,\dots,\alpha_n)$, with $\alpha_i\ge 0$ for all $1\le i\le n$ and $\sum_{i=1}^n \alpha_i=1$,  a \textbf{patch-selection strategy}. Continuing to denote the abundance of the population in the $i$-th patch at time $t \ge 0$ as $\bar X_t^i$, we have $\bar X_t^i =\alpha_i \bar X_t$, where  $\bar X_t=\sum_{i=1}^n \bar X_t^i$ is the total population abundance at time $t \ge 0$.  If we impose these constraints  on
$(\bar X^1 \ldots, \bar X^n)$, then it is heuristically reasonable that
the process $\bar X$ is an autonomous Markov process that satisfies the SDE
\begin{equation}\label{eq:pre}
d\bar X_t = \bar X_t \sum_{i=1}^n \alpha_i \left(\mu_i - \kappa_i \alpha_i \bar X_t\right) \, dt  +\bar X_t \sum_{i=1}^n \alpha_i  \, d E^i_t.
\end{equation}

\begin{remark}\label{R:prelimiting}
One way to justify the formulation of  \eqref{eq:pre} rigorously
is to first modify \eqref{eq:base} to obtain a system of SDEs explicitly accounting for dispersal. Suppose that individuals disperse from patch $i$ to patch $j$ at a rate $\delta d_{ij}$ for some fixed rate matrix $D=(d_{ij})$. As usual, we adopt the convention $d_{ii}=-\sum_{j\neq i} d_{ij}$.
The resulting system of SDEs is
\begin{equation}\label{eq:high_dispersion}
d\tilde X_t^i = \tilde X_t^i \left(\mu_i    -  \kappa_i \tilde X_t^i \right) \, dt
+ \delta \sum_{j } \tilde X_t^j d_{ji} \, dt
+ \tilde X_t^i dE_t^i.
\end{equation}
Assume that the rate matrix $D$ has a unique stationary distribution $\alpha$; that is, $\alpha_j>0$ for $1 \le j \le n$, $\sum_{i=1}^n \alpha_i =1$,
\begin{equation}\label{e:alpha}
\sum_{j=1}^n \alpha_j d_{ji} = 0
\end{equation}
for $1 \le i \le n$.
In this case, a vector $(y^1,\dots,y^n)$ satisfies
\begin{equation}\label{e:y}
\sum_{j=1}^ny^jd_{ji} =0
\end{equation}
for $1 \le i \le n$ if and only if

\begin{equation}\label{e:y2}
y^j=\alpha_j c
\end{equation}
for $1 \le j \le n$ for some constant $c$.  Moreover, summing \eqref{e:y2} we find that
\begin{equation}\label{e:cy}
c=\sum_{i=1}^n y^i.
\end{equation}
Note that by \eqref{e:alpha} we can write the drift term in \eqref{eq:high_dispersion} that contains $\delta$ as
\begin{equation}\label{e:drift}
 \delta \sum_{j } \tilde X_t^j d_{ji} \, dt = \delta \sum_j (\tilde X_t^j-\alpha_j\tilde X_t)d_{ji}\,dt
\end{equation}
where $\tilde X_t :=\sum_{i=1}^n \tilde X_t^i$.
using \eqref{e:y2} and \eqref{e:cy}, we see that
 $(x^1,\dots,x^n)$ and $x:=\sum_{i=1}^n x^i$ are such that
\begin{equation}\label{e:iff}
\sum_{j=1}^n (x^j-\alpha_j x)d_{ji} =0
\end{equation}
for $i=1,\dots,n$ if and only if
\begin{equation}\label{e:cond}
\begin{split}
x^j-\alpha_j x &= \alpha_j \sum_{i=1}^n (x^i-\alpha_i x)\\
&= 0
\end{split}
\end{equation}
for $1 \le j \le n$.

It follows from \eqref{e:drift} and the equivalence between \eqref{e:iff} and \eqref{e:cond} that as  $\delta$ increases the solution of \eqref{eq:high_dispersion}
experiences an increasingly strong drift
towards the one-dimensional subspace
\begin{equation*}
\{(x_1, \ldots, x_n) : x_i = \alpha_i (x_1 + \cdots + x_n), \, i=1,\ldots,n\}.
\end{equation*}
In the limit  $\delta\rightarrow \infty$, it is plausible that
the system \eqref{eq:high_dispersion} converges to one for which
\begin{equation*}
\tilde X_t^i = \alpha_i \tilde X_t,
\end{equation*}
where $\tilde X_t := \tilde X_t^1 + \cdots + \tilde X_t^n$,
and the total population size $\tilde X_t$ satisfies the autonomous
one-dimensionl SDE
\eqref{eq:pre} with $\bar X_t$ replaced by $\tilde X_t$.
This heuristic for obtaining \eqref{eq:pre} as a high dispersal rate
limit of \eqref{eq:high_dispersion} can be made rigorous by applying
Theorem 6.1 from \cite{Katz91}.
\end{remark}

Let $x\cdot y = \sum_{i=1}^n x_i y_i$ denote the standard Euclidean inner product and define another inner product
$\la \cdot,\cdot\ra_\kappa$ by $\la x,y\ra_\kappa:=\sum_{i=1}^n\kappa_i x_i y_i$.  Since $(\alpha \cdot E_t)_{t \ge 0}$ is a Brownian motion with infinitesimal variance parameter $\alpha\cdot\Sigma \alpha$,
 \eqref{eq:pre} can be expressed more simply as
\begin{equation}\label{eq:single}
d\bar X_t = \bar X_t  \left( \alpha\cdot \mu  - \la \alpha, \alpha\ra_\kappa \bar X_t\right) dt  +\bar X_t \sqrt{\alpha\cdot\Sigma \alpha}\, dW_t,
\end{equation}
where $W_t$ is a standard Brownian motion.

The total population $(\bar X_t)_{t\geq 0}$ defined by \eqref{eq:single} behaves {\bf exactly} like the one-patch case defined by  \eqref{e:Z} with the parameters $\mu \to \mu\cdot\alpha$, $\kappa\to \la \alpha, \alpha\ra_\kappa$ and $\sigma\to \sqrt{\alpha\cdot\Sigma \alpha}$. In particular, $(\bar X_t)_{t\geq 0}$ is a diffusion process and we have the following immediate consequence of Proposition~\ref{P:Z}

\begin{proposition}\label{P:one_population_stationary}
Consider the diffusion process $(\bar X)_{t\geq 0}$ given by
\eqref{eq:single}.
\begin{itemize}
\item
If $\bar X_0 = x > 0$, then $\bar X_t > 0$ for all $t \ge 0$ almost surely.
\item
If $\alpha \cdot \mu - \frac{\alpha \cdot \Sigma \alpha}{2}<0$,
then $\lim_{t\to\infty} \bar X_t=0$ almost surely.
\item
If $\alpha \cdot \mu - \frac{\alpha \cdot \Sigma \alpha}{2}=0$,
then $\liminf_{t\to\infty} \bar X_t=0$ almost surely,
$\limsup_{t\to\infty} \bar X_t=\infty$ almost surely,
and
$\lim_{t\rightarrow \infty}\frac{1}{t}\int_0^t \bar X_s\,ds = 0$ almost surely.
\item
If $\alpha \cdot \mu - \frac{\alpha \cdot  \Sigma \alpha}{2} >0$, then
the process $(\bar X_t)_{t\geq 0}$
has a unique stationary distribution $\rho_{\bar X}$ on
$\R_{++}$ with Gamma density
$g(x)= \frac{1}{\Gamma(k)\theta^k}x^{k-1}e^{-\frac{x}{\theta}}$,
where
\[
\theta := \frac{\alpha \cdot \Sigma\alpha}{2  \la \alpha,  \alpha \ra_\kappa}
\mbox{ and }
k := \frac{2\alpha \cdot \mu}{\alpha \cdot \Sigma\alpha}-1.
\]
Moreover,
\[
\lim_{t\to\infty} \frac{1}{t}\int_0^t h(\bar X_s) \, ds = \int_0^\infty h(x)g(x) \, dx \quad \text{almost surely}
\]
for any Borel function $h:\R_{++}\to \R$ with $\int_0^\infty |h(x)|g(x)d\, x<\infty$.
In particular,
\begin{equation*}
\lim_{t\rightarrow \infty}\frac{1}{t}\int_0^t \bar X_s\,ds =
\frac{1}{\la \alpha,\alpha\ra_\kappa} \alpha \cdot \left(\mu - \frac{\Sigma \alpha}{2}\right) \quad \text{almost surely}.
\end{equation*}
\end{itemize}
\end{proposition}

For the dynamics \eqref{eq:base} in patch $i$, Proposition~\ref{P:Z} implies that if there was no coupling between patches
by dispersal, then then population abundance in patch $i$ would converge to $0$ if $\mu_i - \sigma_{ii}/2<0$ and converge to a non-trivial equilibrium if $\mu_i-\sigma_{ii}/2>0$. As noted by \citet{amnat-12} and illustrated below,  the spatially coupled model is such that the population can persist and converge to an equilibrium even when $\mu_i -\sigma_{ii}/2<0$ for all patches.

\subsection*{Persistence of coupled sink populations in symmetric landscapes}
Consider a highly symmetric landscape where
$\mu_i =r$, $\sigma_{ii}=\sigma^2>0$ for all $i$ , $\kappa_i =a $ for all $i$,
 and $\sigma_{ij}=0$ for all $i\neq j$.
 If individuals are equally distributed across the landscape ($\alpha_i=1/n$ for all $i$), then
\[
\mu_i - \frac{\sigma_{ii}}{2}=r -\frac{\sigma^2}{2} \text{ and }\alpha\cdot \mu -\frac{\alpha \cdot \Sigma \alpha}{2}= r - \frac{\sigma^2}{2n}.
\]
The increase in the stochastic growth rate from $r-\sigma^2/2$ for an isolated population to $r-\sigma^2/(2n)$ for the spatially  coupled population stems from individuals spending equal time in patches with uncorrelated environmental fluctuations. Specifically, the environmental variance experienced by individuals distributing their time equally amongst $n$ uncorrelated  patches is $n$ times smaller than the environmental variance experienced by an individual spending their time entirely in one patch. Whenever $\sigma^2>2r>\sigma^2/n$, this reduction in variance allows the entire population to persist despite patches, in and of themselves, not supporting population growth.

\section{Dimorphic model and invasion rates}

To understand the evolution of patch-selection strategies, we now consider competition between populations that only differ in their patch-selection strategy. Let $X_t$ and $Y_t$ be the total population sizes at time $t \ge 0$ of two populations playing the respective patch selection strategies
$\alpha=(\alpha_1,\alpha_2,\dots,\alpha_n)$ and
 $\beta=(\beta_1,\beta_2,\dots,\beta_n)$, so that the densities of the populations in patch $i$ are $\alpha_i X_t$ and $\beta_i Y_t$
 at time $t \ge 0$.
 The dynamics of these two strategies are described
 by the pair of stochastic differential equations
\begin{equation}
\label{dimorphic}
\begin{split}
dX_t &=
X_t \sum_{i=1}^n \alpha_i
\left(\mu_i -\kappa_i(\alpha_i X_t+ \beta_i Y_t)\right) \,dt
+ X_t \sum_{i=1}^n \alpha_i \, dE^i_t\\
dY_t &=  Y_t \sum_{i=1}^n \beta_i
\left(\mu_i  -\kappa_i (\alpha_i X_t + \beta_i Y_t)\right)\,dt
+ Y_t \sum_{i=1}^n \beta_i \, dE^i_t. \\
\end{split}
\end{equation}
Since
\begin{eqnarray*}
d[X,X]_t &=& X_t^2 \alpha\cdot \Sigma \alpha \, dt\\
d[Y,Y]_t &=& Y_t^2 \beta\cdot \Sigma \beta \, dt\\
d[X,Y]_t &=& X_t Y_t \alpha\cdot  \Sigma \beta \, dt,
\end{eqnarray*}
 the diffusion process $((X_t,Y_t))_{t\geq 0}$
 for the spatially coupled, competing strategies  can be represented
 more compactly as
\begin{equation}\label{ex:two}
\begin{aligned}
dX_t &=& X_t\left[\mu\cdot\alpha - \la \alpha,\beta\ra_\kappa Y_t -\la \alpha,\alpha\ra_\kappa X_t \right]\,dt + X_t \sqrt{\alpha\cdot\Sigma\alpha}\,dU_t\\
dY_t &=& Y_t\left[\mu\cdot\beta - \la \alpha,\beta\ra_\kappa X_t -\la \beta,\beta\ra_\kappa Y_t \right]\,dt + Y_t \sqrt{\beta \cdot \Sigma\beta}\,dV_t,
\end{aligned}
\end{equation}
where
$(U,V)$ is a (non-standard) Brownian motion with
covariance structure $d[U,U]_t = dt$, $d[V,V]_t = dt$, and
$d[U,V]_t = \frac{\alpha\cdot  \Sigma \beta}
{ \sqrt{\alpha\cdot \Sigma\alpha}\sqrt{\beta\cdot \Sigma\beta}}\,dt$. Using a construction similar from Remark~\ref{R:prelimiting}, system \eqref{ex:two} can be seen as a high dispersal limit. This system exhibits a degeneracy when $U=V$ i.e.  $ \frac{\alpha\cdot \Sigma \beta} { \sqrt{\alpha\cdot\Sigma\alpha}\sqrt{\beta\cdot\Sigma\beta}}=1$.  If $\Sigma$ is nonsingular, then, by the Cauchy-Schwarz inequality, this degeneracy only occurs if $\alpha = \beta$. We do not consider this possibility in what follows.

To determine whether the two populations coexist or one displaces the other, we introduce the invasion rate $\I(\alpha,\beta)$ of a population playing strategy $\beta$ when introduced at small densities into a resident population playing strategy $\alpha$. As shown in the next proposition, this invasion rate is defined by linearizing the dynamics of $Y$ and computing the long-term population growth rate $\I(\alpha, \beta)$ associated with this linearization.   When $\I(\alpha,\beta)>0$, the population playing strategy $\beta$ tends to increase when rare. When $\I (\alpha,\beta)<0$, the population playing strategy $\beta$ tends to decrease when rare.

\begin{proposition}
\label{P:linearize_invader}
Consider the partially linearized system
\begin{equation}\label{ex:linear}
\begin{aligned}
d \bar X_t &=& \bar X_t\left[\mu\cdot\alpha -\la \alpha,\alpha\ra_\kappa \bar X_t \right]\,dt + \bar X_t \sqrt{\alpha\cdot\Sigma\alpha}\,dU_t\\
d \hat Y_t &=& \hat Y_t\left[\mu\cdot\beta - \la \alpha,\beta\ra_\kappa \bar X_t  \right]\,dt + \hat Y_t \sqrt{\beta \cdot \Sigma\beta}\,dV_t.
\end{aligned}
\end{equation}
Assume  $\bar X_0>0$ and $\hat Y_0>0$.

If $\alpha \cdot( \mu -  \Sigma\alpha/2) > 0$,
so the Markov process $\bar X$ has a stationary distribution
concentrated on $\R_{++}$, then the limit
$\lim_{t\rightarrow\infty} \frac{\log \hat Y_t}{t}$
exists almost surely and is given by
\begin{equation}\label{e_L_M}
\I (\alpha,\beta)
=
\beta\cdot(\mu   - \Sigma\beta/2)
- \frac{\la \alpha,\beta\ra_\kappa }{\la \alpha,\alpha\ra_\kappa} \alpha\cdot\left(\mu - \Sigma \alpha/2\right).
\end{equation}

On the other hand, if $\alpha \cdot( \mu -  \Sigma\alpha/2)  \le 0$, so that
$\lim_{t \to \infty} \frac{1}{t} \int_0^t \bar X_s \, ds = 0$ almost surely,
then the limit
$ \lim_{t\rightarrow\infty} \frac{\log \hat Y_t}{t}$
exists almost surely and is given by
\begin{equation}\label{e_L_M2}
\I (\alpha,\beta)=\beta \cdot (\mu -  \Sigma \beta/2).
\end{equation}
\end{proposition}

\begin{proof}
By It\^o's lemma,
\[
d \log \hat Y_t = \left(\mu\cdot \beta-\la \alpha, \beta \ra_\kappa \bar X_t\right)\,dt + \sqrt{\beta\cdot \Sigma\beta}\,dV_t - \frac{1}{2} (\beta\cdot \Sigma\beta) \,dt.
\]

Assume that $\mu\cdot \alpha -  \frac{\alpha \cdot \Sigma \alpha}{2}>0$.  By Proposition \ref{P:one_population_stationary},
\[
\lim_{t\rightarrow \infty}\frac{1}{t}\int_0^t \bar X_s \,ds
 = \frac{\alpha }{2\la \alpha,\alpha\ra_\kappa}\cdot (2\mu - \Sigma \alpha) \quad \text{almost surely}.
\]
Therefore,
\begin{equation*}
\lim_{t\rightarrow\infty}\frac{\log \hat Y_t}{t}=\beta\cdot\mu - \frac{\la \alpha,\beta\ra_\kappa }{2\la \alpha,\alpha\ra_\kappa}\alpha\cdot (2\mu - \Sigma \alpha)
  - \frac{1}{2}\beta\cdot\Sigma\beta \quad \quad \text{almost surely},
\end{equation*}
as claimed.

On the other hand,
assume that $\mu\cdot \alpha -  \frac{\alpha \cdot \Sigma \alpha}{2} \le 0$.
By Proposition \ref{P:one_population_stationary},
\[
\lim_{t\rightarrow \infty}\frac{1}{t}\int_0^t \bar X_s \,ds
 = 0 \quad \text{almost surely}.
\]
Therefore,
\begin{equation*}
\lim_{t\rightarrow\infty}\frac{\log \hat Y_t}{t}=\mu\cdot \beta - \frac{1}{2}\beta\cdot\Sigma\beta \quad \text{almost surely},
\end{equation*}
again as claimed.
\end{proof}

In the next proposition, we show that if a population playing strategy $\beta$ cannot invade a population playing strategy $\alpha$ (i.e. $\I (\alpha,\beta)<0$), then the population strategy $\alpha$ can invade the population playing strategy $\beta$ (i.e. $\I(\beta,\alpha)>0$). This suggests, as we will show in the next section, that such a strategy $\alpha$ should exclude strategy $\beta$.

\begin{proposition}\label{P:Lyapunov_exponents}
Suppose that $\alpha \cdot (\mu -  \Sigma\alpha/2) > 0$
and $\I(\alpha,\beta)<0$.  Then, $\I(\beta,\alpha)> 0$.
\end{proposition}

\begin{proof}
Set $A:=\alpha \cdot( \mu -  \Sigma\alpha /2)$
and $B := \beta \cdot( \mu -  \Sigma\beta /2)$.
Assume that $A > 0$ and $\I(\alpha,\beta)<0$.
To show that $\I (\beta,\alpha)>0$,
we consider two cases, $B  \le 0$ and $B>0$. Suppose $B \le 0$.
Then, $\I(\beta,\alpha)= A > 0$ by Proposition~\ref{P:linearize_invader}
and by assumption.

Alternatively, suppose that $B >0$. Then
\[
\I(\alpha,\beta)
=
B - A \frac{\la \alpha, \beta\ra_\kappa}{\la \alpha, \alpha\ra_\kappa}
\text{  and }
\I(\beta,\alpha)
=
A - B \frac{\la \alpha, \beta\ra_\kappa}{\la \beta, \beta\ra_\kappa}
\]
by Proposition~\ref{P:linearize_invader}.
Assume, contrary to our claim, that $\I(\alpha,\beta) <0$ and $\I(\beta,\alpha) \leq 0$.
From the Cauchy-Schwarz inequality
$\la x,y \ra_\kappa \leq \la x,x \ra_\kappa^{1/2} \la y,y \ra_\kappa^{1/2}$ we get
\[
\begin{split}
& B \la \alpha, \alpha\ra_\kappa ^{1/2} \la \beta, \beta \ra_\kappa ^{1/2}
\geq B \la \alpha, \beta \ra_\kappa \geq A\la \beta, \beta\ra_\kappa\\
& A \la \alpha, \alpha\ra_\kappa ^{1/2} \la \beta, \beta \ra_\kappa ^{1/2}
\geq A \la \alpha, \beta \ra_\kappa > B\la \alpha,\alpha \ra_\kappa . \\
\end{split}
\]
The above inequalities yield the contradiction
$B \la \alpha, \alpha\ra_\kappa ^{1/2}\geq A \la \beta,\beta\ra _\kappa^{1/2}$
and  $B \la \alpha, \alpha\ra_\kappa ^{1/2}< A \la \beta, \beta\ra_\kappa ^{1/2}$.
\end{proof}

An immediate consequence of Proposition~\ref{P:linearize_invader} is the following corollary. This corollary implies that if a population playing strategy $\beta$ can invade a population playing strategy $\alpha$ and a population playing strategy $\alpha$ can invade a population playing strategy $\beta$, then a single population playing strategy $\alpha$ converges to
a non-trivial equilibrium and the same is true of a single population playing strategy $\beta$.
This suggests, as we show in the next section,  that under these conditions these two strategies should coexist.

\begin{corollary}\label{C:Lyapunov_exponents_positive}
The invasion rate satisfies $\I(\alpha,\beta)\le \beta\cdot(\mu - \Sigma\beta/2)$. In particular, if $\I(\alpha,\beta)>0$ and $\I(\beta,\alpha)>0$,
then $\alpha \cdot( \mu -  \Sigma\alpha /2) > 0$
and $\beta \cdot( \mu -  \Sigma\beta /2) > 0$.
\end{corollary}

\section{Exclusion and protected polymorphisms}

Our main results about the dimorphic process $(X,Y)$ is that the invasion rates determine the long-term fate of competing strategies. If the invasion rates predict that strategy $\beta$ cannot invade a population playing strategy $\alpha$, then the population playing strategy $\alpha$ drives the population playing strategy $\beta$ asymptotically to extinction as shown in the following theorem. We give a proof in Appendix B.

\begin{theorem}\label{T:stationary_concentration}
If $\alpha\cdot(\mu -  \Sigma\alpha/2) > 0$ and $\I(\alpha,\beta)<0$, then, for $(x,y) \in \R_{++}^2$, the probability measures
\begin{equation*}
\frac{1}{t} \int_0^t \mathbb{P}^{(x,y)}\{(X_s,Y_s) \in \cdot\} \, ds
\end{equation*}
converge weakly as $t \to \infty$ to
$\rho_{\bar X} \otimes \delta_0$, where $\rho_{\bar X}$ is the unique
stationary distribution of $\bar X_t$ concentrated on $\R_{++}$, and $\delta_0$ is the point mass at $0$.
\end{theorem}

On the other hand, if the invasion rates predict that each strategy can invade when rare, then the following theorem proves that the competing strategies coexist: for any initial conditions the joint distribution of $(X_t,Y_t)$ converges as $t\rightarrow \infty$ to a probability distribution on $\R_{++}^2$ with density $\psi$ and, moreover, for any Borel set $B\subset\R_{++}^2$ the long term proportion of times $t$ for which $(X_t,Y_t)$ spends in $B$ converges to
\[
\int_B \psi(x,y) \,dx dy.
\]
A proof is given in Appendix C. In order to appreciate the assumptions of the theorem, it helps to
recall Corollary \ref{C:Lyapunov_exponents_positive} which says that if  $\I(\alpha,\beta)>0$ and $\I(\beta,\alpha)>0$ then $\alpha \cdot( \mu -  \Sigma\alpha /2) > 0$
and $\beta \cdot( \mu -  \Sigma\beta /2) > 0$ so that a single population playing strategy $\alpha$ or $\beta$ will persist.

\begin{theorem}\label{T:stationary_positive_Lyapunov}
Suppose that  $\I(\alpha,\beta)>0$ and $\I(\beta,\alpha)>0$. Then, there exists a unique stationary distribution $\pi$ of $(X,Y)$ on $\R_{++}^2$ that is absolutely continuous with respect to Lebesgue measure.
Moreover, for any bounded, measurable function $f:\R_{++}^2\to \R$
\begin{equation}\label{e_pathwise_limit}
\lim_{t\rightarrow \infty} \frac{1}{t} \int_0^t f(X_s, Y_s) \,ds
=
\int_{\R_{++}^2} f(x,y) \, \pi(dx,dy)
\quad \text{almost surely}.
\end{equation}
Furthermore, the process $(X,Y)$ is strongly ergodic, so that for any initial distribution $q$ one has
\begin{equation}\label{e_ergodic}
\lim_{t\rightarrow \infty}
d_{\mathrm{TV}}(\Pr^q\{(X_t,Y_t) \in \cdot\}, \pi) =0,
\end{equation}
where $d_{\mathrm{TV}}$ is the total variation distance.
\end{theorem}

From the perspective of population genetics, the coexistence of these two strategies corresponds to a \textbf{protected polymorphism}: each strategy (a morph) increases when rare and, therefore, is protected from extinction. This protection from extinction, however, is only ensured over ecological time scales as mutations may result in new morphs that can displace one or both coexisting morphs~\citep{ravigne-etal-04}. The concept of protected polymorphisms was introduced by ~\citet{prout-68} when studying deterministic models of competing haploid populations in a spatially heterogenous with overlapping generations. \citet{turelli-etal-01} extended this concept to stochastic difference equations for competing haploid populations with a constant population size. Theorem~\ref{T:stationary_positive_Lyapunov} provides a mathematically rigorous characterization of protected polymorphisms for our stochastic models with fluctuating population sizes.

Theorem~\ref{T:stationary_positive_Lyapunov} implies that coexistence depends on the intrinsic stochastic growth rate of the populations and the competitive effect of each population on the other. The intrinsic stochastic growth rates are given by
\[
r_\alpha= \alpha\cdot(\mu - \Sigma \alpha/2) \mbox{ and } r_\beta =\beta\cdot (\mu - \Sigma \beta/2).
\]
While the competitive effect of the population with strategy $\alpha$ on the population with strategy $\beta$ is given by the ratio of the magnitude of $\alpha$ projected in the $\beta$ direction (i.e. $\langle \beta/ \|\beta\|_\kappa, \alpha\rangle_\kappa$ where $\|\beta\|_\kappa = \sqrt{\langle \beta,\beta\rangle_\kappa}$) divided by the magnitude of $\beta$ (i.e. $\|\beta \|_\kappa$). Mathematically, the competitive effect of $\alpha$ on $\beta$ and the competitive effect of $\beta$ on $\alpha$ are given by
\[
C_{\alpha,\beta} = \frac{\langle \beta/ \|\beta\|_\kappa, \alpha\rangle_\kappa}{\|\beta\|_\kappa}
\mbox{ and }
C_{\beta,\alpha} = \frac{\langle \alpha/ \|\alpha\|_\kappa, \beta \rangle_\kappa}{\|\alpha\|_\kappa}.
\]
Provided $r_\alpha$ and $r_\beta$ are positive, Theorem~\ref{T:stationary_positive_Lyapunov} implies that there is a protected polymorphism if
\begin{equation}\label{eq:protected}
\frac{r_\alpha}{r_\beta}> C_{\beta,\alpha} \mbox{ and } \frac{r_\beta}{r_\alpha} >C_{\alpha,\beta}.
\end{equation}
In words, the relative intrinsic stochastic growth rate of each population must exceed the competitive effect  on itself due to the other population. Conversely, if one of the inequalities in \eqref{eq:protected} is reversed, then Theorem~\ref{T:stationary_concentration} implies that one population excludes the other. Unlike the standard Lotka-Volterra competition equations, Proposition~\ref{P:Lyapunov_exponents} implies that both inequalities in \eqref{eq:protected} cannot be simultaneously reversed and, consequently, bistable dynamics are impossible.

\subsection*{Environmental stochasticity impedes protected polymorphisms in symmetric landscapes.} Consider a landscape where all patches have the same carrying capacities (e.g. $\kappa_i=1$ for all $i$), the same intrinsic rates of growth (i.e. $\mu_i  =a $ for all $i$), and the same amount of uncorrelated environmental stochasticity (e.g. $\sigma_{ii}=\sigma^2$ for all $i$ and $\sigma_{ij}=0$ for $i\neq j$). Then the protected polymorphism inequalities~\eqref{eq:protected} become
\begin{equation}\label{eq:protected2}
\frac{a-\sigma^2\|\alpha\|^2/2}{a-\sigma^2\|\beta\|^2/2}> C_{\beta,\alpha} \mbox{ and } \frac{a-\sigma^2\|\beta\|^2/2}{a-\sigma^2\|\alpha\|^2/2} >C_{\alpha,\beta}
\end{equation}
where the only $\sigma^2$ dependency is on the left hand sides of both inequalities. As $\frac{a- \sigma^2 \|\beta\|^2/2}{r - \sigma^2\| \alpha\|^2 /2}$ is a decreasing function of $\sigma^2$ whenever $\frac{\|\alpha\|}{\|\beta\|}<1$ and an increasing function of $\sigma^2$ whenever $\frac{\|\alpha\|}{\|\beta\|}>1$, it follows that the set of set of strategies supporting a protected polymorphism
\[
A(\sigma^2)=\{(\alpha,\beta): \eqref{eq:protected2} \mbox{ holds}\}
\]
is a decreasing function of $\sigma^2$ i.e $A(\sigma_2^2)$ is a proper subset of $A(\sigma_1^2)$ whenever $\sigma_2>\sigma_1\ge 0$. Figure~\ref{F:symmetric}A illustrates this conclusion for a two-patch landscape. Intuitively, increasing environmental stochasticity in these symmetric landscapes reduces the stochastic growth rate for all strategies and, thereby, makes it less likely for populations to persist let alone coexist. For asymmetric landscapes, how the set $A(\sigma^2)$ of protected polymorphisms varies with $\sigma^2$ is more subtle, as illustrated in Figure~\ref{F:symmetric}B. In this case, some protected polymorphisms are facilitated by environmental stochasticity, while other protected polymorphisms are disrupted by environmental stochasticity.

\begin{figure}
\begin{center}
\includegraphics[width=3in]{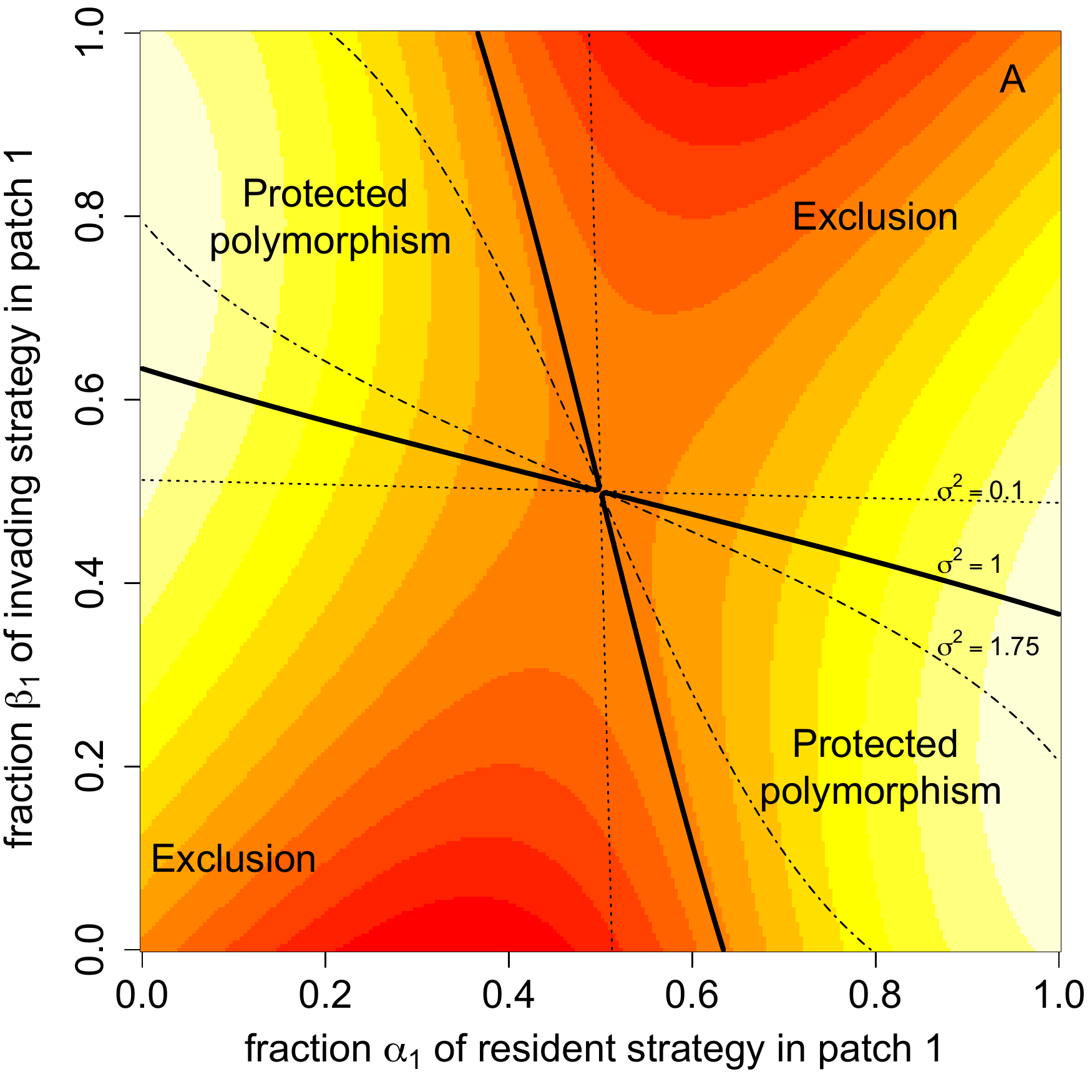}\includegraphics[width=3in]{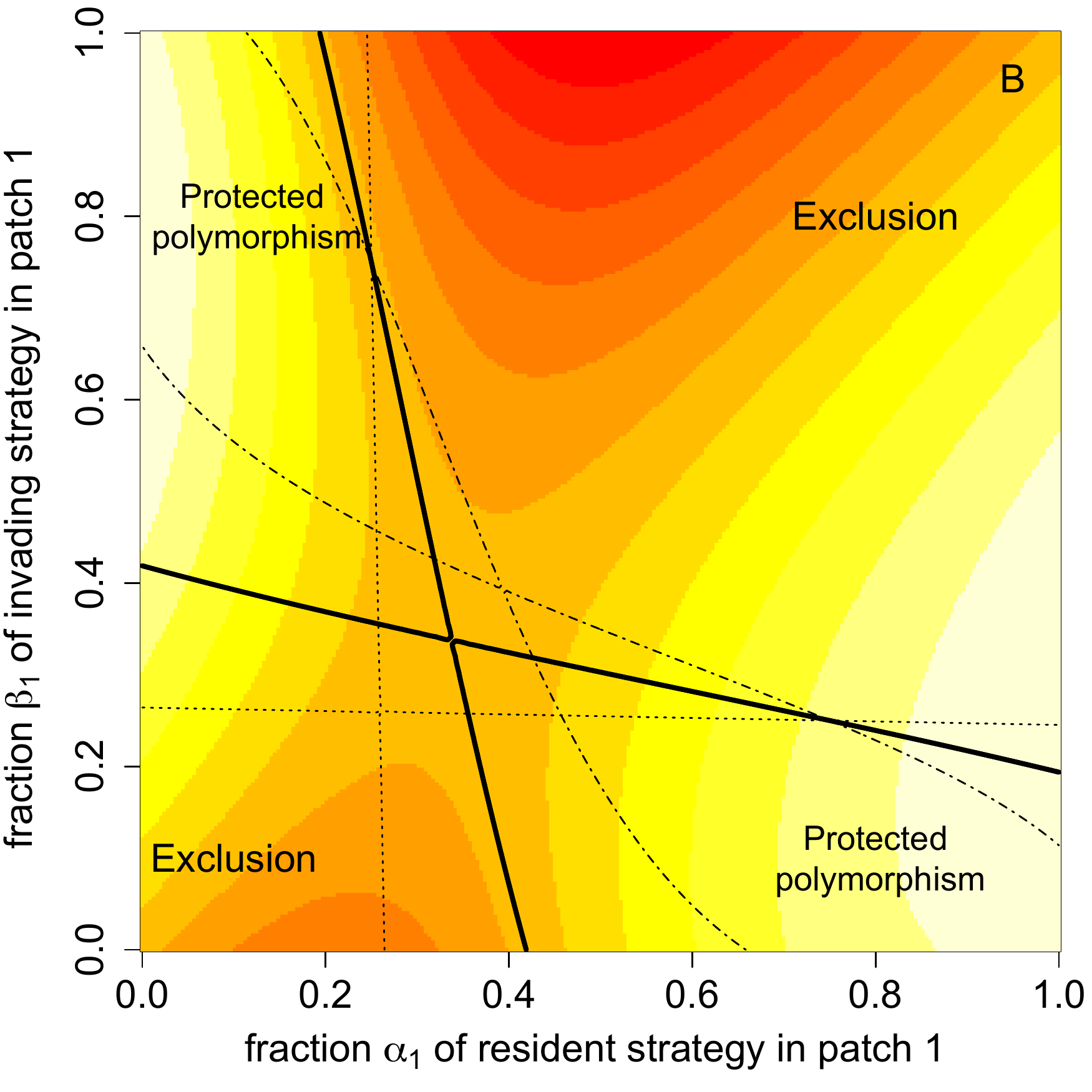}
\end{center}
\caption{Protected polymorphisms and exclusion in two-patch landscapes. Contour plots of $\I(\alpha,\beta)$ where lighter shades correspond to higher values of $\I(\alpha,\beta)$.  The regions where $\I(\alpha,\beta)\I(\beta,\alpha)>0$ are delineated by the solid curves and  correspond to parameter combinations supporting a protected polymorphism. Regions where $\I(\alpha,\beta)\I(\alpha,\beta)<0$ correspond to strategies that cannot coexist. The dashed-dotted and dotted curves indicate how regions of coexistence and exclusion change for higher and lower levels of environmental stochasticity $\sigma^2$, respectively. In panel A, the landscape is spatially homogeneous with $\mu=(1,1)$, $\kappa=(1,1)$ and $\Sigma=\sigma^2 I $ where $I$ is the $2\times 2$ identity matrix. In B, the landscape is spatial heterogeneous with respect to the deterministic carrying capacities $\kappa=(3,1)$ and the remaining parameters as A. }\label{F:symmetric}
\end{figure}

For the symmetric landscapes, we can identify a strategy that displaces all others. Namely, the strategy $\alpha=(\frac{1}{n},\dots,\frac{1}{n})$ of visiting all patches with equal frequency. This strategy maximizes the function function  $\alpha\mapsto a-\sigma^2 \|\alpha \|^2/{2}$. Hence, if we consider a competing strategy $\beta \neq \alpha$, then $\alpha\cdot \beta = \| \alpha\|^2 = \frac{1}{n}$ and
\[
\begin{aligned}
\I(\alpha,\beta) &= a- \sigma^2\|\beta\|^2/2 - \frac{\alpha\cdot \beta}{\|\alpha \|^2} \left( a- \frac{\sigma^2 \|\alpha\|^2}{2} \right)\\
&= a- \sigma^2\|\beta\|^2/2 - \left( a- \frac{\sigma^2 \|\alpha\|^2}{2} \right)<0.
\end{aligned}
\]
e.g. the invasion rates are negative along the vertical transect $\alpha_1=1/2$ in Figure.~\ref{F:symmetric}A. This strategy $\alpha$ is an example of an evolutionarily stable strategy that we discuss further in the next section.

\section{Evolutionarily stable strategies}

The concept of an evolutionary stable strategy was introduced by \citet{maynardsmith-price-73}. Loosely stated, an evolutionary strategy is a strategy that cannot be invaded by any other strategy and, consequently, can be viewed as an evolutionary endpoint. For our models, we say patch selection strategy $\alpha$ is an \emph{evolutionarily stable strategy (ESS)} if $\I(\alpha,\beta)<0$ for all strategies $\beta \neq \alpha$. In light of Theorem~\ref{T:stationary_concentration}, an ESS not only resists invasion attempts by all other strategies, but can displace all other strategies. An ESS $\alpha$ is called a \emph{pure ESS} if $\alpha_i=1$ for some patch $i$, otherwise it is a \emph{mixed ESS}. Our next result provides an algebraic characterization of mixed and pure ESSs. However, it remains to be understood whether these ESSs can be reached by small mutational steps in the strategy space (i.e. are \textbf{convergently stable}~\citep{geritz-etal-97}).

\begin{theorem}\label{T:ESS}
Assume that the covariance matrix $\Sigma$ is positive definite and that there is at least one patch selection strategy which persists in the absence of competition with another strategy; that is, that $\max_\alpha \alpha\cdot( \mu -  \Sigma \alpha/2)>0$.
\begin{description}
\item[Mixed strategy] An ESS $\alpha$ with $\alpha_i>0$ for $i\in I$ with $I\subseteq\{1,2,\dots,n\}$ and $|I|\ge 2$ satisfies
\begin{equation}\label{E:ESS}
-\frac{\alpha \cdot \Sigma \alpha}{2}  = \mu_i -  \kappa_i \alpha_i \frac{\alpha\cdot (2\mu - \Sigma \alpha) }{2\la \alpha,\alpha\ra_\kappa} -\sum_{j=1}^n \sigma_{ij} \alpha_j
\end{equation}
for all $i\in I$. Conversely, if $|I|=n$, then a strategy $\alpha$ satisfying \eqref{E:ESS} is an ESS.
\item [Pure strategy] The strategy $\alpha_i=1$ and $\alpha_j=0$ for $j\neq i$ is an ESS if and only if
\begin{equation}\label{E:ESS2}
\mu_j -\frac{\sigma_{jj}}{2}< -\frac{\sigma_{jj}}{2} + \sigma_{ij} -\frac{\sigma_{ii}}{2}
\end{equation}
for all $j\neq i$.
\end{description}
Furthermore, in the case of $n=2$, there exists a mixed ESS whenever the reversed inequalities
\begin{equation}\label{E:ESS3}
\mu_j -\frac{\sigma_{jj}}{2}> -\frac{\sigma_{jj}}{2} + \sigma_{ij} -\frac{\sigma_{ii}}{2}
\end{equation}
hold for $i=1,2$ and $j\neq i$.
\end{theorem}

The first statement of Theorem~\ref{T:ESS} provides a sufficient and necessary condition for a mixed ESS utilizing all patches. For example, in a symmetric landscape (as described in the previous section), this ESS condition is only satisfied for $\alpha=(1/n,1/n,\dots,1/n)$.

The second statement of Theorem~\ref{T:ESS} provides a characterization of when using only a single patch is an ESS. Since the right hand side of equation \eqref{E:ESS2} is negative, using patch $i$ can only be an ESS if all other patches have a negative stochastic rate of growth, $\mu_j-\frac{\sigma_{jj}}{2}<0$ for all $j\neq i$. However, even if only patch $i$ has a positive stochastic growth rate, an ESS may use the other patches, as we illustrate next for two-patch landscapes. \vskip 0.1in

\subsection*{ESSs in two-patch, uncorrelated landscapes}  For an uncorrelated two patches landscape (i.e. $n=2$ and $\sigma_{12}=0$), Theorem~\ref{T:ESS} implies that there is a mixed ESS whenever
\begin{equation}\label{2patch1}
\mu_1 > -\sigma_{22}/2 \text{ and } \mu_2 > -\sigma_{11}/2
\end{equation}
and this ESS satisfies
\begin{equation}\label{2patch2}
\alpha_i = \frac{\mu_i + \alpha \cdot \Sigma\alpha}{\kappa_i (\mu \cdot \alpha - \alpha \cdot \Sigma\alpha/2)/\la \alpha,\alpha \ra_\kappa+\sigma_{ii}}.
\end{equation}
Equation~\eqref{2patch1} implies that even if deterministic growth in patch $2$ is strictly negative (i.e. $\mu_2<0$), then there is selection for movement into this patch provided the variance of the fluctuations in patch $1$ are sufficiently large relative to the intrinsic rate of decline in patch $2$ (Fig.~\ref{F:source-sink}).

\begin{figure}
\includegraphics[width=0.9\textwidth]{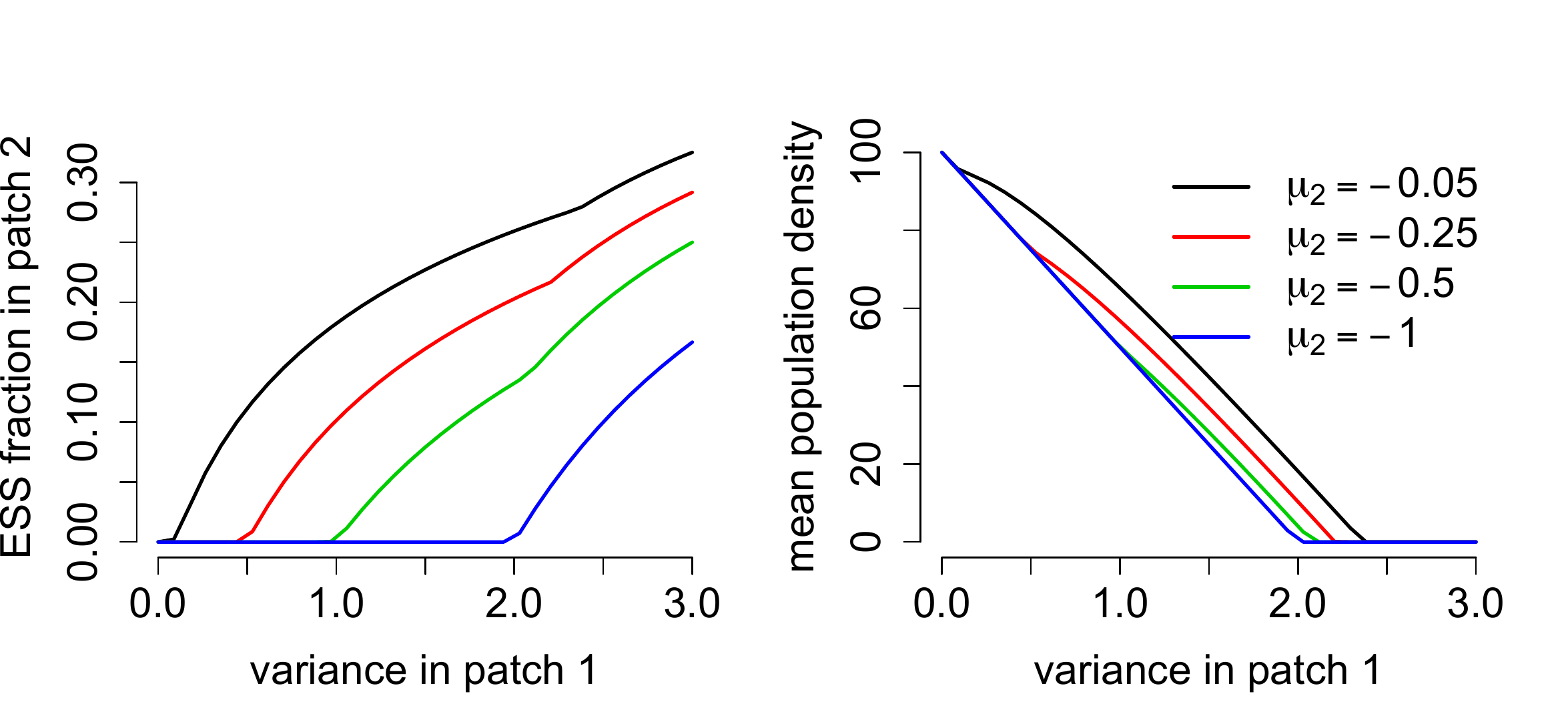}
\caption{ESS for patch selection (left) and mean population abundance (right) in a source-sink landscape.  Parameter values: $n=2$, $\sigma_{11}=\sigma^2$, $\sigma_{22}=\sigma_{12}=0$, $\mu=(1,\mu_2)$, and $\kappa=(1,1)$. }\label{F:source-sink}
\end{figure}

In the limit of no noise (i.e. $\sigma_{ii}\downarrow 0$ for $i=1,2$),  equation~\eqref{2patch2}  becomes
\[
\alpha_i = \frac{\mu_i}{\kappa_i  \mu \cdot \alpha/\la \alpha,\alpha \ra_\kappa}.
\]
While our results do not apply to the deterministic case, this limiting expression for the ESS suggests, correctly, that the ESS for the deterministic model satisfies
\[
\alpha_i = \frac{\mu_i /\kappa_i}{\sum_j \mu_j /\kappa_j} \text{ whenever } \mu_i>0.
\]
In other words, the fraction of individuals selecting patch $i$ is proportional to the equilibrium density $\mu_i/\kappa_i$ supported by patch $i$. Equation~\eqref{2patch2} implies that adding stochasticity in equal amounts to all patches (i.e. $\sigma_{ii}=\sigma^2$ for all $i$) results in an ESS where, relative to the deterministic ESS, fewer individuals select  patches supporting the highest mean population abundance and more individuals selecting patches supporting lower mean population abundances (Fig.~\ref{F:K}).

\begin{figure}
\includegraphics[width=0.9\textwidth]{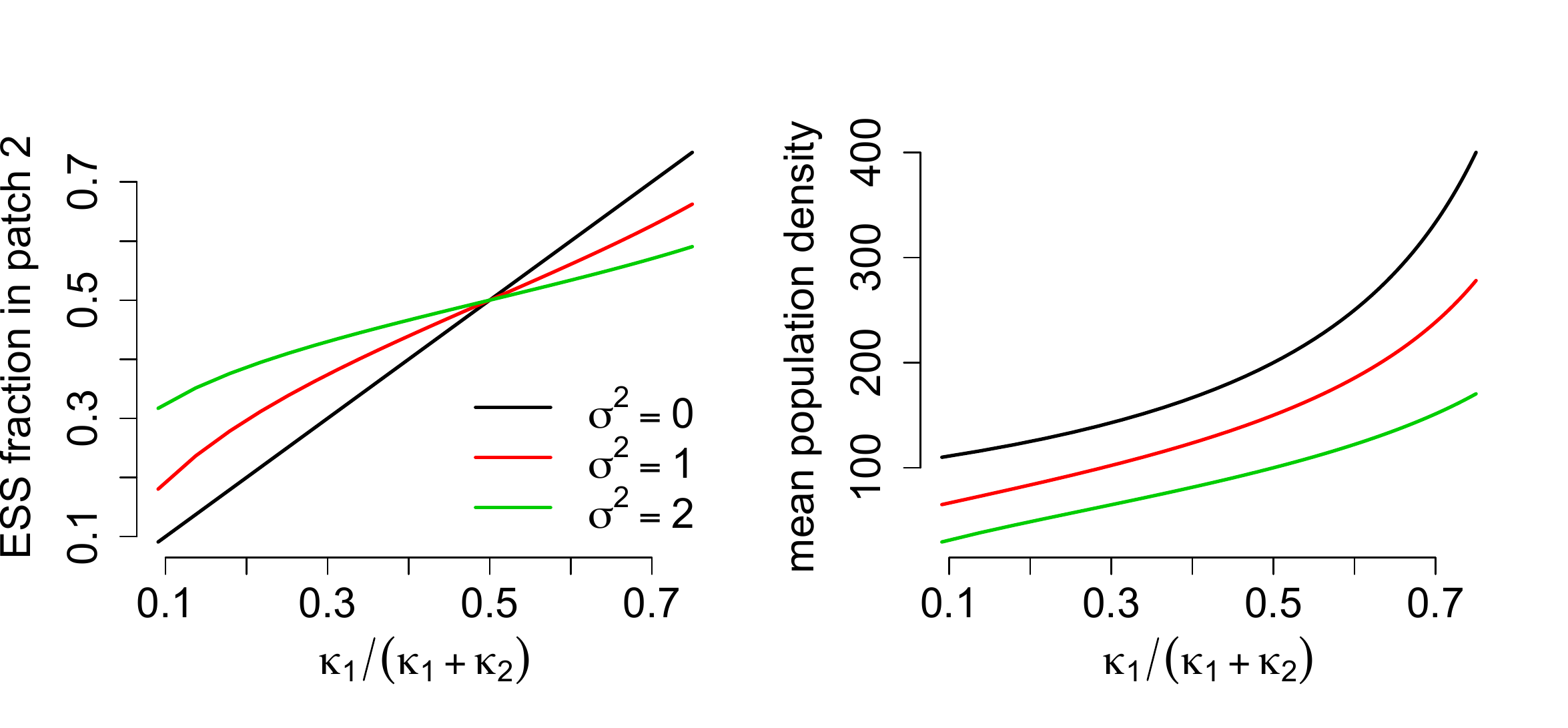}
\caption{The effect of the deterministic carrying capacities  and environmental stochasticity on the ESS for patch selection (left) and mean population abundance (right) in a two-patch landscape. The ratio $\kappa_1/(\kappa_2+\kappa_2)$ corresponds to the ratio of the deterministic carrying capacity ($\mu_2/\kappa_2$) in patch $2$ to the sum of the deterministic carrying capacities ($\mu_1/\kappa_2+\mu_2/\kappa_2$) when $\mu_1=\mu_2=1$.   Parameter values: $n=2$, $\sigma_{11}=\sigma_{22}=\sigma^2$, $\sigma_{12}=0$, $\mu=(1,1)$, and $\kappa=(1,\kappa_2)$. }\label{F:K}
\end{figure}

\section{Discussion}
Habitat selection by organisms is a complex process determined by a mixture of genetic, developmental, ecological, and environmental factors. For ecologists, habitat selection plays a fundamental role in determining the spatial and temporal distribution of a population~\citep{rosenzweig-81,orians-wittenberger-91}. For evolutionary biologists, habitat selection determines the suite of environmental factors driving local adaptation~\citep{edelaar-bolnick-12}. Indeed, in the words of the eminent evolutionary biologist Ernst Mayr, ``With habitat and food selection -- behavioral phenomena -- playing a major role in the shift into new adaptive zones, the importance of behavior in initiating new evolutionary events is self-evident'' \citep[p. 604]{mayr-63}. Here, we examined how spatial and temporal heterogeneity in demographic rates across multiple habitat patches  influence the dynamics of competing populations who only differ in their habitat patch selection preferences. We assume that habitat selection has a genetic basis (e.g.  genes that influence the physiological or neurological capacity of individuals to detect and respond to habitat cues) and that genetic differences in habitat choice have no pleiotropic effects on habitat specific fitness.  Our analysis reveals that, generically, only two outcomes are possible, coexistence or displacement of one population by the other for all initial conditions, and that these outcomes are determined by the invasion rates of populations when rare. In addition to providing a mathematically rigorous justification of prior work, our analysis provides new insights into protected polymorphisms for habitat selection and raises several questions about evolutionary stable strategies  for habitat selection.

Protected polymorphisms correspond to populations of competing genotypes exhibiting negative frequency-dependence: each population tends to increase when rare~\citep{prout-68}. In the case of patch selection, these competing populations differ in the frequencies in which they select habitat patches.  In a survey of the empirical literature, \citet{jaenike-holt-91} found ``that genetic variation for habitat selection is common, especially in arthropods and mollusks, the groups that have been studied most frequently.'' Moreover, they argued that some of this variation may be maintained through protected polymorphism. Specifically, ``in a haploid model without intrinsic fitness differences among genotypes [i.e. soft selection], genetic variation in fixed habitat preferences may be maintained stably'' \citep[pg. S83]{jaenike-holt-91}. We provide a general analytic criterion (see, inequality \eqref{eq:protected}) characterizing these protected polymorphisms for spatially and temporally	variable environments. This criterion depends on the intrinsic fitnesses ($r_\alpha$ and $r_\beta$) of each population and their competitive coefficients ($C_{\alpha,\beta}$ and $C_{\beta,\alpha}$) that characterize the effect of each population on the other. Competitive effects are greatest when there is an overlap in patch use and one population tends to select the patches with the higher carrying capacities more than the other population. Intuitively, by occupying patches with a larger carrying capacities, populations achieve higher regional densities. Coupled with overlap in patch use, these higher densities result in a greater competitive impact of one population on another.  A protected polymorphism occurs when the relative fitness of each population (e.g. $r_\alpha/r_\beta$ for strategy $\alpha$) is greater than the competitive effect of the other population on it (e.g. $r_\alpha/r_\beta>C_{\beta,\alpha}$ for the population playing strategy $\alpha$). Hence, as in the case of species coexistence~\citep{chesson-00}, protected polymorphism are most likely when fitness differences are small (i.e. $r_\alpha/r_\beta \approx 1$) and competitive effects are small (i.e. both $C_{\alpha,\beta}$ and $C_{\beta,\alpha}<1$). Environmental stochasticity solely  effects the intrinsic fitness terms and can facilitate or inhibit protected polymorphisms. For landscapes in which all patches experience the same degree of uncorrelated, temporal variation, environmental stochasticity has an inhibitory effect  as it magnifies fitness differences between competing strategies (e.g. $r_\alpha/r_\beta$ increases with environmental stochasticity). For asymmetric landscapes, however, temporal variability can facilitate polymorphisms by reducing fitness differences of competing strategies.

In contrast to protected polymorphisms, our analysis reveals that populations playing an evolutionarily stable strategy (ESS) for patch selection not only thwart invasion attempts by all other strategies but also can invade and displace a population playing any other strategy. Furthermore, our analysis provides a mathematically rigorous justification of an earlier characterization of ESSs~\citep{amnat-12}. This characterization implies that populations playing the ESS always occupy source habitats (i.e. patches where $\mu_i -\sigma_{ii}^2/2>0$). Indeed, consider a population playing strategy $\alpha$ that does not occupy some source patch, say patch $i$. Then a different behavioral genotype $\beta$ that only selects patch $i$ can invade as $\I(\alpha,\beta)=\mu_i -\sigma_i^2/2>0$. In the limiting case of a deterministic environment, our characterization of the ESS recovers the classic result of \citet{mcpeek-holt-92}: the fraction of time spent in a patch is proportional to the carrying capacity of the patch.  Adding environmental stochasticity generally results in populations playing the ESS decreasing the time spent in the patches with larger carrying capacities and possibly making use of sink patches (i.e. patches where $\mu_i -\sigma^2_i/2<0$). This shift in patch choice can be viewed as a spatial form of  bet hedging: individuals increase fitness by decreasing the variance in their stochastic growth rate at the expense of their mean growth rate~\citep{childs-etal-10}.

We are able to show that for two patch landscapes there always exists an ESS for patch selection. However, several questions remain unanswered. First, what happens for landscapes with more than two patches? Is there always an ESS? Second, while we know that a population playing an ESS can displace a monomoprhic population playing a different strategy, can it displace polymorphic populations? Finally, are ESSs always convergently stable~\citep{geritz-etal-97}? If there are positive answers to this final suite of questions, then ESSs can be generally viewed as the ultimate evolutionary end state for patch selection strategies.

Going beyond the models considered here, studying the evolution habitat use faces many challenges. Our models assume that populations spend a fixed fraction of time in each patch and do so instantaneously. What happens if we relax  these assumptions? For example, if populations are more ideal and able to track changes in population density instantaneously, then we have something closer to the classical notion of ideal free movement~\citep{fretwell-lucas-69}. For these populations, what is the optimal (in an evolutionary sense) density-dependent strategy? Moreover, can such a strategy displace the static strategies considered here? Alternatively, if populations are less ideal and diffusing randomly on the landscape, what happens then? The linear version of this question was tackled in part by \citet{jmb-13}. However, the mathematical analysis for analogous stochastic models with density-dependent feedbacks is largely unexplored. Going beyond single species, the coevolution of patch selection among interacting species has a rich history for spatially heterogeneous, but temporally homogeneous environments~\citep{vanbaalen-sabelis-93,krivan-97,amnat-00,vanbaalen-etal-01,eer-02,cressman-etal-04,prsb-06,cantrell-etal-07}. For example, spatial heterogeneity can select for the evolution of contrary choices in which the prey prefers low quality patches to escape the predator and the predator prefers high quality patches to capture higher quality food items~\citep{fox-eisenbach-92,amnat-00}. Understanding how environmental stochastic influences this coevolution of patch choice and the community level consequences of these coevolutionary outcomes provides a plethora of important, yet largely untouched challenges for future work.

\section*{Appendix A: Proof of Proposition~\ref{P:Z}}
The stochastic differential equation for $Z$
is of the form
\begin{equation}
\label{SDE_reminder}
d Z_t = b(Z_t) \, dt + \sigma(Z_t) \, dW_t,
\end{equation}
where $b(z) :=\mu z - \kappa z^2$ and $\sigma(z) := \sigma z$.
It follows from
It\^o's existence and uniqueness theorem for strong solutions of
stochastic differential equations that this equation has a unique
strong solution up to possibly a finite but strictly positive
explosion time.

Set $R_t := \log Z_t$ for $t \ge 0$.  By It\^o's lemma,
\begin{equation}
\label{log_SDE}
d R_t = \left(\mu - \frac{\sigma^2}{2} - \kappa \exp(R_t)\right) \, dt + \sigma \, dW_t.
\end{equation}
It follows from
the comparison principle of Ikeda and Watanabe (see Chapter VI Theorem 1.1
of \cite{IW89}), Theorem 1.4 of \cite{LeG83},
or Theorem V.43.1 of \cite{RW00}) that
\begin{equation}
\label{E:compare_BM}
R_t \le R_0 + \left(\mu - \frac{\sigma^2}{2}\right)t + \sigma W_t,
\end{equation}
and so $Z$ does not explode to $+\infty$ in finite time.  Moreover, since
$r \mapsto \mu - \kappa e^r$ is a
bounded, uniformly Lipschitz function on $(-\infty,0]$ it follows from
It\^o's existence and uniqueness theorem that $R$ does not explode to $-\infty$
in finite time, so that $Z$ does not hit $0$ in finite time.  We could have
also established this result by using the scale function and speed measure
calculated below to check Feller's necessary and sufficient for
the boundary point of a one-dimensional diffusion to be inaccessible --
see Theorem 23.12 of \cite{K02}.

It is not hard to check using It\^o's lemma that an explicit solution of
the SDE is
\[
Z_t = \frac{Z_0 \exp((\mu-\sigma^2/2)t+\sigma W_t)}{1+Z_0 \frac{\mu}{\kappa}\int_0^t \exp((\mu-\sigma^2/2)s+\sigma W_s) \, ds}.
\]

We see from the inequality \eqref{E:compare_BM} that if
$\mu - \sigma^2/2 < 0$, then $\lim_{t \to \infty} Z_t = 0$ almost surely.

We use the theory based on
the scale function and speed measure of a one-dimensional diffusion
(see, for example, Chapter 23 of \cite{K02} or Sections V.6-7 of \cite{RW00})
below to establish that $Z$ is positive recurrent
with a unique stationary distribution when $\mu - \sigma^2/2 > 0$.  Similar
calculations show that $Z$ is null recurrent when $\mu - \sigma^2/2 = 0$,
and hence $\liminf_{t \to \infty} Z_t = 0$ almost surely and
$\limsup_{t \to \infty} Z_t = \infty$.
It follows from \eqref{log_SDE} and the comparison
principle that if $Z'$ and $Z''$ are two solutions of
\eqref{SDE_reminder} with respective parameters $\mu',\kappa',\sigma'$
and $\mu'',\kappa'',\sigma''$ satisfying $\mu' \le \mu''$, $\kappa'=\kappa''$, $\sigma' = \sigma''$
and the same initial conditions, then $Z_t' \le Z_t''$.
We will show below that
\[
\lim_{t \to \infty} \frac{1}{t} \int_0^t Z_s \, ds
=
\frac{1}{\kappa} \cdot (\mu - \sigma^2/2)
\]
almost surely when $\mu - \sigma^2/2 > 0$, and hence
\[
\lim_{t \to \infty} \frac{1}{t} \int_0^t Z_s \, ds = 0
\]
almost surely when $\mu - \sigma^2/2 = 0$.

We now identify the scale function and speed measure of the one-dimensional
diffusion $Z$.
A choice for the scale function is
\begin{equation}
\label{e_scale}
\begin{split}
s(x)
& = \int_c^x \exp\left(-\int_{a}^y \frac{2b(z)}{\sigma^2(z)}\,dz\right)\,dy \\
& = \int_c^x \left(\frac{y}{a}\right)^{-2\mu/\sigma^2} e^{\frac{2\kappa}{\sigma^2}(y-a)}\,dy \\
\end{split}
\end{equation}
for arbitrary $a,c \in \R_{++}$ (recall that the scale function is only defined
up to affine transformations).
If we set $\tilde \sigma = (\sigma s')\circ s^{-1}$, then
\begin{equation*}
d s(Z_t) = \tilde \sigma(s(Z_t)) \, d\tilde W_t
\end{equation*}
and the diffusion process $s(Z)$ is in natural scale on the state space
$s(\R_{++})$ with speed measure
$m$ that has density $\frac{1}{\tilde \sigma^2}$.

The total mass of the speed measure is
\begin{eqnarray}\label{e_mass_speed_measure}
m(\R_{++})
&=& \int_{s(\R_{++})} \frac{1}{\tilde \sigma^2(x)}\,dx
=  \int_{s(\R_{++})} \frac{1}{((\sigma s')\circ s^{-1})^2(x)}\,dx
= \int_0^\infty \frac{1}{\sigma^2(u)s'(u)}\,du\nonumber\\
&=& \int_0^\infty  \frac{1}{(\sigma u)^2\left(\frac{u}{a}\right)^{-2\mu/\sigma^2} e^{\frac{2\kappa}{\sigma^2}(u-a)}}\,du\nonumber\\
&=& \frac{1}{\sigma^2a^{2\mu/\sigma}} \int_0^\infty u^{\frac{2\mu}{\sigma^2} -2} e^{-\frac{2\kappa}{\sigma^2}(u-a)}\,du.
\end{eqnarray}
By Theorem 23.15 of \cite{K02},
the diffusion process $Z$ has a stationary distribution
concentrated on $\R_{++}$
if and only if the process $s(Z)$ has $(-\infty,+\infty)$
as its state space and the
speed measure has finite total mass
or $s(Z)$ has a finite interval as its
state space and the boundaries are reflecting.
The introduction of an extra negative drift to geometric Brownian motion
cannot make zero a reflecting boundary, so we are interested in conditions
under which
$s(\R_{++}) = (-\infty,\infty)$ and the speed measure has finite total mass.
We see from \eqref{e_scale} and \eqref{e_mass_speed_measure}
that this happens if and only if $\mu-\sigma^2/2 >0$, a condition we assume
holds for the remainder of the proof.

The diffusion $s(Z)$ has a stationary distribution with density
$f:=\frac{1}{m(\R_{++})\tilde \sigma ^2}$ on $s(\R_{++}) = (-\infty,+\infty)$,
and so the stationary distribution of $Z$ is the distribution on
$\R_{++}$ that has density
\begin{eqnarray*}
g(x) &=& f(s(x)) s'(x)\\
&=& \frac{1}{m(\R_{++})\tilde \sigma ^2(s(x))} s'(x)\\
&=& \frac{1}{m(\R_{++})\sigma^2(x)s'(x)}\\
&=& \frac{1}{m(\R_{++})x^2\sigma^2\left(\frac{x}{a}\right)^{-2\mu/\sigma^2} e^{\frac{2\kappa}{\sigma^2}(x-a)}},
\quad x\in \R_{++}.
\end{eqnarray*}
This has the form
of a $\mathrm{Gamma}(k,\theta)$ density with parameters
$\theta :=\frac{\sigma^2}{2\kappa}$ and  $k=\frac{2\mu}{\sigma^2}-1$. Therefore,
\begin{equation*}
g(x)
=
\frac{1}{\Gamma(k)\theta^k}x^{k-1}e^{-\frac{x}{\theta}}
=
\frac{1}{\Gamma\left(\frac{2\mu}{\sigma^2}-1\right)\left(\frac{\sigma^2}{2\kappa}\right)^{\frac{2\mu}{\sigma^2}-1}}x^{\frac{2\mu}{\sigma^2}-2}e^{\frac{-2\kappa x}{\sigma^2}},
\quad x \in \R_{++}.
\end{equation*}

Theorem 20.21 from \cite{K02} implies that the shift-invariant $\sigma$-field is trivial for all starting points.
The ergodic theorem for stationary stochastic processes then tells us that, if we start $Z$ with its stationary distribution,
\begin{equation*}
\lim_{t\rightarrow \infty}\frac{1}{t}\int_0^t  h(Z_s)\,ds  = \int_0^\infty h(x) g(x) \, dx
\end{equation*}
for any Borel function $h:\R_{++}\to \R$ with
$\int_0^\infty |h(x)|g(x) \, dx<\infty$. Since $Z$ has positive continuous transition densities we can conclude that
\begin{equation*}
\lim_{t\rightarrow \infty}\frac{1}{t}\int_0^t  h(Z_s)\,ds  = \int_0^\infty h(x) g(x) \, dx
\end{equation*}
$\Pr^x$-almost surely for any $x\in\R_{++}$.

In particular,
\begin{equation*}
\int_{\R_{++}}x g(x) \, dx
= k \theta
= \frac{1}{\kappa} \cdot \left(\mu - \frac{\sigma^2}{2}\right).
\end{equation*}


\section*{Appendix B: Proof of Theorem~\ref{T:stationary_concentration}}

To simplify our presentation,  we re-write the joint dynamics of $X$ and $Y$ as
\begin{eqnarray}\label{e_MN_simplified}
dX_t &=&X_t \left(\mu\cdot\alpha  - (aX_t + cY_t)\right) \,dt + \sigma_X X_t \,dU_t\\
dY_t &=&Y_t  \left(\mu\cdot\beta  -  (cX_t + bY_t)\right) \,dt + \sigma_Y Y_t \,dV_t,\nonumber
\end{eqnarray}
where $a:=\la\alpha,\alpha\ra_\kappa$, $b :=\la\beta,\beta\ra_\kappa$,
$c:=\la\alpha,\beta\ra_\kappa$, $\sigma_X:=\sqrt{\alpha\cdot \Sigma\alpha}$,
and $\sigma_Y:=\sqrt{\beta\cdot \Sigma\beta}$.

To prove Theorem~\ref{T:stationary_concentration}, we need several preliminary results. First, we prove
existence and uniqueness of solutions to the system \eqref{e_MN_simplified} as well as a useful comparison result in  Theorem~\ref{T:comparison}. Second, in Proposition~\ref{p:positivity}, we establish that $(X_t,Y_t)$ remains in $\mathbb{R}_{++}^2=(0,\infty)^2$ for all $t\ge 0$ whenever $(X_0,Y_0)\in \mathbb{R}_{++}^2$. Third, in Proposition~\ref{P:stationary_Kurtz}, we show that weak limit points of the empirical measures $\frac{1}{t}\int_0^t \mathbb{P}^{(x,y)}\{(X_s,Y_s)\in \cdot \} \, ds$ are stationary distributions for the process $(X,Y)$ thought of as a process on
$\R_+^2$ (rather than $\R_{++}^2$). Finally, we show that $\lim_{t\to\infty} Y_t =0$ with probability one in Proposition~\ref{P:M extinction} and conclude by showing that $\frac{1}{t}\int_0^t  \mathbb{P}^{(x,y)}\{(X_s,Y_s)\in \cdot \} \, ds$ converges weakly to $\rho_{\bar X} \otimes \delta_0$ concentrated on
$\R_{++} \times \{0\}$.

\begin{theorem}\label{T:comparison}
The stochastic differential equation
 in \eqref{e_MN_simplified} has a unique strong solution and $X_t,Y_t\in L^p(\mathbb{P}^{(x,y)})$ for all $t,p>0$ for all $(x,y) \in \mathbb{R}_{++}^2$.
This solution satisfies $X_t > 0$ and $Y_t > 0$ for all $t \ge 0$,
$\mathbb{P}^{(x,y)}$-almost surely for all $(x,y) \in \mathbb{R}_{++}^2$.
 Let $((\bar X_t ,\bar Y_t))_{t \ge 0}$ be the stochastic process defined by
 the pair of stochastic differential equations
\begin{eqnarray}\label{e:linearized}
d\bar X_t &=& \bar X_t \left(\mu\cdot\alpha - a \bar X_t \right) \,dt + \sigma_X \bar X_t \,dU_t\\
d\bar Y_t &=& \bar Y_t \left(\mu\cdot\beta -   b\bar Y_t\right) \,dt + \sigma_Y \bar Y_t \,dV_t\nonumber
\end{eqnarray}
If $(X_0,Y_0)=(\bar X_0,\bar Y_0)$, then
\begin{equation*}
X_t\leq\bar X_t
\end{equation*}
and
\begin{equation*}
Y_t\leq\bar Y_t
\end{equation*}
for all $t\geq 0$.
\end{theorem}
\begin{proof}
The uniqueness and existence of strong solutions is fairly standard,
see, for example, Theorem 2.1 in \cite{LM09}. One notes that the drift
coefficients are locally Lipschitz so strong solutions exist
and are unique up to the explosion time. It is easy to show this explosion time is almost surely infinite (see Theorem 2.1 in \cite{LM09}).
Next, suppose that $X_0 = \bar X_0$.
We adapt the comparison principle of Ikeda and Watanabe  (Chapter VI Theorem 1.1
from \cite{IW89}) proved by the local time techniques of Le Gall (see Theorem 1.4 from \cite{LeG83} and Theorem V.43.1 in \cite{RW00}) to show that
$\bar X_t - X_t \ge 0$ for all $t \ge 0$.

Define $\rho:\R_+\rightarrow \R_+$ by $\rho(x)=|x|^2$. Note that

\[
\begin{split}
& \int_0^t \rho(|\bar X_s-X_s|)^{-1} \ind{\{\bar X_s - X_s>0\} }\, d[\bar X- X]_s \\
& \quad = \int_0^t \rho(|\bar X_s-X_s|)^{-1}
(\sigma_X \bar X_s - \sigma_X X_s)^2
\ind\{\bar X_s - X_s>0\}\,ds\\
&\quad \leq \sigma_X^2 t.\\
\end{split}
\]

Since $\int_{0+}\rho(u)^{-1}\,du=\infty$, by Proposition V.39.3
from \cite{RW00} the local time at 0 of $X-\bar X$ is zero for all
$t\geq 0$.
Put $x^+ := x \vee 0$.
By Tanaka's formula (see equation IV.43.6 in \cite {RW00}),
\begin{eqnarray*}
(X_t-\bar X_t)^+ &=& \int_0^t \ind{\{X_s-\bar X_s>0\}}
(\sigma_X X_s - \sigma_X \bar X_s)\,dU_t\\
&~& + \int_0^t  \ind{ \{X_s-\bar X_s>0\}}
\left[(\mu\cdot\alpha  - (aX_s + cY_s))X_s
-(\mu\cdot\alpha  - a\bar X_s)\bar X_s\right]\,ds.
\end{eqnarray*}
For $K>0$ define the stopping time
\begin{equation*}
T_K:=\inf\{t>0 : X_t\geq K ~\text{or} ~\bar X_t \geq K~\}
\end{equation*}
and the stopped processes $X^K_t=X_{T_K\wedge t}$ and
$\bar X^K_t=\bar X_{T_K\wedge t}$.
Then, stopping the processes at $T_K$ and taking expectations yields
\[
\begin{split}
0
& \leq
\E (X^K_t-\bar X^K_t)^+ \\
& = \E \int_0^{t\wedge T_K}  \ind\{X_s-\bar X_s>0\}
 \left[(\mu\cdot\alpha X_s - X_s(aX_s + cY_s))
 -(\mu\cdot\alpha \bar X_s - a\bar X_s^2 )\right]\,ds\\
& =\E \int_0^{t\wedge T_K}   \ind\{X_s-\bar X_s>0\} \left[ \mu\cdot\alpha (X_s - \bar X_s) -a(X_s^2-\bar X_s^2) -cX_sY_s \right]\,ds\\
& \leq
\E  \int_0^{t\wedge T_K}   \ind\{X_s-\bar X_s>0\} \mu\cdot\alpha (X_s - \bar X_s)\,ds\\
& \leq
\mu\cdot\alpha\, \E\int_0^{t\wedge T_K}  (X_s - \bar X_s)^+\,ds\\
& \leq  \mu\cdot\alpha\, \E\int_0^t (X^K_s - \bar X^K_s)^+\,ds. \\
\end{split}
\]
By Gronwall's Lemma (see, for example, Appendix 5 of \cite{EK05})
$\E [(X^K_t-\bar X^K_t)^+] = 0$ for all $t\geq 0$, so $X^K_t\leq\bar X^K_t$ for all $t\geq 0$. Now let $K\rightarrow \infty$ and recall that $\bar X$ does not explode to get that  $X_t\leq\bar X_t$ for all $t\geq 0$. Since we have shown before that $\bar X$ is dominated by a geometric Brownian motion,
a process that has finite moments of all orders, we get that $X_t,Y_t\in L^p(\mathbb{P}^{(x,y)})$ for all $t,p>0$ and for all $(x,y) \in \mathbb{R}_{++}^2$.
\end{proof}

\begin{remark}
Note that the SDEs for all the processes considered here have unique strong solutions in $L^p$ for all $t\geq 0, p>0$ and for all strictly positive starting points. This follows by arguments similar to those that are in Theorem 2.1 from \cite{LM09} and in Theorem \ref{T:comparison} by noting that our SDEs for $(X,Y)$, $(\bar X, \bar Y)$ etc. are all of the form
\begin{eqnarray*}
d\breve X_t &=& \breve X_t\left[\lambda_1 - \lambda_2 \breve Y_t -\lambda_3 \breve X_t \right]\,dt + \breve X_t \sigma_X\,dU_t\\
d\breve Y_t &=& \breve Y_t\left[\lambda_4 - \lambda_5 \breve X_t -\lambda_6 \breve Y_t \right]\,dt + \breve Y_t \sigma_Y\,dV_t\\
\breve X_0 &=& x\\
\breve Y_0&=& y
\end{eqnarray*}
for $\lambda_1,\dots,\lambda_6\in\R_+$ and $x,y\in\R_{++}$.
\end{remark}

The next proposition tells us that none of our processes hit zero in finite time.

\begin{proposition}\label{p:positivity}
Let $(X,Y)$ be the process given by \eqref{e_MN_simplified}. If $(X_0,Y_0) \in \R_{++}^2$, then
$(X_t,Y_t) \in \R_{++}^2$ for all $t \ge 0$ almost surely.
A similar conclusion holds for all of the other processes we work with.
\end{proposition}

\begin{proof}
As an example of the method of proof, we look at the process $(X,Y)$ given by \eqref{e_MN_simplified}.
Taking logarithms and using It\^o's lemma,
\begin{equation*}
d \log X_t
= \left(\mu\cdot\alpha  - (aX_t + cY_t)-\frac{1}{2}\sigma_X^2\right) \,dt
+ \sigma_X  \,dU_t.
\end{equation*}
Therefore,
\begin{equation*}
\log X_t
= \int_0^t\left(\mu\cdot\alpha  - (aX_s + cY_s)-\frac{1}{2}\sigma_X^2\right) \,ds
+ \sigma_X U_t.
\end{equation*}
can't go to $-\infty$
in finite time because $X_t$ and $Y_t$ do not blow up.
\end{proof}

\begin{proposition}\label{P:stationary_Kurtz}
Let $(X,Y)$ be the process given by \eqref{e_MN_simplified} and fix $(x,y) \in \R_{++}^2$.
Any sequence $\{t_n\}_{n \in \N}$ such that $t_n\rightarrow\infty$
has a subsequence $\{u_n\}_{n \in \N}$ such that the sequence of probability measures
\begin{equation*}
 \frac{1}{u_n} \int_0^{u_n} \mathbb{P}^{(x,y)} \{(X_s,Y_s) \in \cdot \} \, ds
\end{equation*}
converges in the topology of weak convergence of probability measures
on $\R_+^2$.  Any such limit is a stationary distribution
for the process $(X,Y)$ thought of as
a process with state space $\R_+^2$.
\end{proposition}

\begin{proof}
Set $\varphi(x,y) :=x+y$ so that $\varphi \geq 0$ for $x,y > 0$.
Put $\psi(x,y)=\mu\cdot\alpha x + \mu\cdot\beta y  -x(ax+cy)-y(cx+by)$.
Note that $\psi$ is bounded above on the quadrant $x,y \geq 0$ and
$\lim_{\|(x,y)\|\rightarrow\infty}\psi(x,y)=-\infty$ where $\|\cdot\|$ is the Euclidean distance on $\R^2$.
Using It\^o's lemma we get
\begin{eqnarray*}
\varphi(X_t,Y_t)- \int_0^t \psi(X_s,Y_s)\,ds &=& \int_0^t\sigma_Y Y_s\,dV_s
+\int_0^t\sigma_X X_s\,dU_s.
\end{eqnarray*}
Therefore, $\varphi(X_t,Y_t)- \int_0^t \psi(X_s,Y_s)\,ds$ is a martingale.
Applying Theorem 9.9 of \cite{EK05} completes the proof.
\end{proof}

The following result is essentially Theorem 10 in \cite{LWW11}.
We include the proof for completeness.

\begin{proposition}
\label{P:M extinction}
Suppose that $\alpha\cdot \mu - \alpha \cdot \Sigma \alpha/2>0$, $\beta\cdot \mu -\beta \cdot \Sigma \beta/2>0$,  and  $\I(\alpha,\beta)<0$. If $(X,Y)$ is the process given by \eqref{e_MN_simplified}, then $\lim_{t\rightarrow \infty} Y_t = 0$ $\mathbb{P}^{(x,y)}$-a.s.
for all $(x,y) \in \R_{++}^2$.
\end{proposition}

\begin{proof}
Using Ito's lemma and the definition of $\I(\alpha,\beta)$,
\begin{eqnarray*}
a\frac{\log\left(\frac{Y_t}{Y_0}\right)}{t} - c\frac{\log\left(\frac{X_t}{X_0}\right)}{t} &=& a\left(\mu\cdot\beta -\frac{\sigma_Y^2}{2}\right) - c \left(\mu\cdot\alpha -\frac{\sigma_X^2}{2}\right) -(ab-c^2)\frac{\int_0^t Y_s\,ds}{t} \\
&~&+ ~a\sigma_Y \frac{V_t}{t} -c\sigma_X\frac{U_t}{t}\\
&=& a \I(\alpha,\beta) - (ab-c^2)\frac{\int_0^t Y_s\,ds}{t} + a\sigma_Y \frac{V_t}{t} -c\sigma_X\frac{U_t}{t}.
\end{eqnarray*}

By the Cauchy-Schwarz inequality,
$(ab-c^2)=\la\alpha,\alpha\ra _\kappa \la \beta,\beta\ra _\kappa - (\la \alpha,\beta\ra _\kappa)^2
\geq 0$, and so
\begin{eqnarray*}
\frac{\log\left(\frac{Y_t}{Y_0}\right)}{t} &\leq& \frac{c}{a}\frac{\log\left(\frac{X_t}{X_0}\right)}{t} + \I(\alpha,\beta) + \sigma_Y \frac{V_t}{t} - \frac{c}{a}\sigma_X\frac{U_t}{t}.
\end{eqnarray*}

Let $\bar X$ be the process defined by \eqref{e:linearized} with $\bar X_0=X_0$.  Proposition \ref{P:one_population_stationary} implies
\begin{equation}
\label{e_integral}
\lim_{t\rightarrow \infty }\frac{1}{t}\int_0^t \bar X_s\,ds =(\mu\cdot \alpha -\sigma_X^2/2)/a \quad \text{almost surely}.
\end{equation}
It follows from Theorem \ref{T:comparison} that
$X_t\leq \bar X_t$ for all $t\geq 0$. Thus, with probability one,
\begin{eqnarray*}
\limsup_{t\rightarrow\infty}\frac{\log X_t}{t} &\leq& \limsup_{t\rightarrow\infty}\frac{\log \bar X_t}{t} \\
&=& \left(\mu\cdot \alpha - \frac{\sigma_X^2}{2}\right) - a\lim_{t\rightarrow \infty}\frac{1}{t}\int_0^t \bar X_s\,ds + \sigma_X \lim_{t\rightarrow \infty}\frac{U_t}{t}\\
&=& \left(\mu\cdot \alpha - \frac{\sigma_X^2}{2}\right) - a (\mu\cdot \alpha -\sigma_X^2/2)/a\\
&=& 0.
\end{eqnarray*}
Since $U$ and $V$ are Brownian motions,
$\lim_{t\rightarrow \infty} \frac{U_t}{t}
= \lim_{t\rightarrow \infty} \frac{V_t}{t}=0$,
and $\limsup_{t\rightarrow \infty} \frac{\log X_t}{t}\leq 0$ almost surely,
so
\begin{equation*}
\limsup_{t\rightarrow \infty} \frac{\log Y_t}{t} \leq \I(\alpha,\beta) <0 \quad \text{almost surely}.
\end{equation*}
In particular, $\lim_{t \to \infty} Y_t = 0$ almost surely.
\end{proof}

We can now finish the proof of Theorem \ref{T:stationary_concentration}.
Fix $\epsilon>0$ and $\eta>0$ sufficiently small.
Define the stopping time
\begin{equation*}
T_{\epsilon}:=\inf \{t \ge 0 :Y_t\geq \epsilon\}.
\end{equation*}
and the stopped process $X^\epsilon_t:=X_{t\wedge T_{\epsilon}}$.
By Proposition \ref{P:M extinction}, there exists $T>0$ such that
\begin{equation*}
\Pr^{(x,y)}\{Y_t\leq \epsilon ~\text{for all}~ t\geq T\}\geq 1-\eta
\end{equation*}
Define the process $\check X$ via
\begin{equation*}
d\check X_t = \check X_t [(\mu\cdot\alpha-c\epsilon)-a\check X_t]\,dt +  \sigma_X \check X_tdU_t
\end{equation*}
and the stopped process  $\check X^\epsilon_t:=\check X_{t\wedge T_{\epsilon}}$.
Start the process $\check X$ at time $T$ with the condition $\check X_T = X_T$.  We want to show that the process $\check X^\epsilon$ is dominated by the process $X^\epsilon$, that is $X^\epsilon_t\geq \check X^\epsilon_t$ for all $t\geq T$. By the strong Markov property, we can assume $T=0$.

The proof is very similar to the one from Theorem \ref{T:comparison}.
With the notation from the proof of Theorem \ref{T:comparison}, we have
\[
\begin{split}
& \int_0^t \rho(|\check X^\epsilon_s-X^\epsilon_s|)^{-1} \ind\{\check X_s^\epsilon-X_s^\epsilon>0\} \, d[\check X^\epsilon-  X^\epsilon]_s \\
& \quad =\int_0^t \rho(|\check X^\epsilon_s-X^\epsilon_s|)^{-1}
(\sigma_X \check X^\epsilon_s - \sigma_X X^\epsilon_s)^2 \ind\{\check X_s^\epsilon-X_s^\epsilon>0\}]\,ds \\
&\quad \leq \sigma_X^2 t \\
\end{split}
\]
so the local time of the process $\check X^\epsilon-X^\epsilon$ at zero is identically zero.
Then, using Tanaka's formula,
\small
\begin{eqnarray*}
(\check X^\epsilon_t-X^\epsilon_t)^+ &=& \int_0^{t\wedge T_\epsilon} \ind\{\check X_s-X_s>0\}(\sigma_X \check X_s - \sigma_X X_s)\,dU_t\\
&~& + \int_0^{t\wedge T_\epsilon}  \ind\{\check X_s-X_s>0\} \left[((\mu\cdot\alpha-c\epsilon)\check X_s-a\check X_s^2)-(\mu\cdot\alpha X_s -   X_s(cY_s + aX_s))\right]\,ds.
\end{eqnarray*}
\normalsize
Taking expectations,
\begin{eqnarray*}
\E[(\check X^\epsilon_t-X^\epsilon_t)^+] &=& \E\int_0^{t\wedge T_\epsilon}\ind\{\check X_s-X_s>0\}[(\mu\cdot\alpha(\check X_s-X_s)-(c\epsilon\check X_s -cX_sY_s)\\
&~&~ - a(\check X_s^2-X_s^2))\,ds]\\
&\leq& \mu\cdot\alpha \,\E\int_0^{t\wedge T_\epsilon} (\check X_s-X_s)^+\,ds\\
&\leq& \mu\cdot\alpha  \,\E\int_0^t (\check X^\epsilon_s-X^\epsilon_s)^+\,ds.
\end{eqnarray*}

By Gronwall's Lemma,
$\E[(\check X^\epsilon_t-X^\epsilon_t)^+]=0$.
As a result, remembering we assumed $T=0$,
we have $\check X^\epsilon_t\leq X^\epsilon_t$ for all $t\geq T$.
For $\epsilon$ small enough we know that
$\check X$ has a stationary distribution concentrated on $\R_{++}$.
For any sequence $a_n\rightarrow \infty$, if the Cesaro averages
$\frac{1}{a_n} \int_0^{a_n} \mathbb{P}^{(x,y)}\{(X_s,Y_s) \in \cdot\} \, ds$
converge weakly, then the limit is a distribution of the form
$\varphi \otimes \delta_0$, where $\varphi$ is a mixture of
the unique stationary distribution $\rho_{\bar X}$ described in Proposition~\ref{P:one_population_stationary}  and the point mass at $0$.
By the above,
the limit of $\frac{1}{a_n} \int_0^{a_n} \mathbb{P}^{(x,y)}\{(X_s,Y_s) \in \cdot\} \, ds$ cannot have any mass at $(0,0)$ because $\check X_t\leq X_t$ on the event $\{Y_t\leq \epsilon ~\text{for all}~ t\geq T\}$ that has probability $\Pr^{(x,y)}\{Y_t\leq \epsilon ~\text{for all}~ t\geq T\}\geq 1-\eta$. Since $\eta>0$ was arbitrary, we conclude that
$\varphi=\rho_{\bar X} \otimes \delta_0$, as required.

\section*{Appendix C: Proof of Theorem~\ref{T:stationary_positive_Lyapunov}}

Our proof is along the same lines as the proofs of Theorems 4 and 5
in \cite{SBA11}.

We will once again simplify our notation by re-writing the SDE for the
pair $(X,Y)$ as in \eqref{e_MN_simplified}.  We assume throughout this
appendix that the hypotheses of Theorem~\ref{T:stationary_positive_Lyapunov}
hold; that is, $\I(\alpha,\beta) > 0$ and $\I(\beta,\alpha) > 0$.

Let $((\bar X_t ,\bar Y_t))_{t \ge 0}$ be the stochastic process defined by
 the pair of stochastic differential equations in \eqref{e:linearized}
with initial conditions $(\bar X_0,\bar Y_0) = (X_0, Y_0)$.  We know from
Theorem~\ref{T:comparison} that $X_t \le \bar X_t$ and $Y_t \le \bar Y_t$
for all $t \ge 0$.

Note from Corollary~\ref{C:Lyapunov_exponents_positive} that
$\alpha \cdot (\mu - \Sigma \alpha/2) > 0$ and
$\beta \cdot (\mu - \Sigma \beta/2) > 0$ and hence,
by Proposition~\ref{P:one_population_stationary},
the process $(\bar X, \bar Y)$ has a unique stationary distribution
on $\R_{++}^2$ and is strongly ergodic.

Let
\begin{equation*}
\Pi_t(\cdot):= \frac{1}{t} \int_0^t \ind\{(X_s, Y_s)\in\cdot\} \,ds
\end{equation*}
be the normalized occupation measures of $(X,Y)$.
We know that the random probability measures
\begin{equation*}
\bar\Pi_t(\cdot):= \frac{1}{t} \int_0^t \ind\{(\bar X_s, \bar Y_s)\in\cdot\} \,ds
\end{equation*}
converge almost surely and so, in particular, they are tight on $\R_+^2 = [0,\infty)^2$;
that is, for any $\epsilon>0$ we can find a box $[0,K] \times [0,K]$ such that
\begin{equation*}
\frac{1}{t}\int_0^t \ind\{(\bar X_s, \bar Y_s) \in [0,K] \times [0,K]\} \, ds
> 1-\epsilon~\text{ for all}~ t > 0.
\end{equation*}

Therefore,
\begin{eqnarray*}
\frac{1}{t}\int_0^t \ind\{(X_s, Y_s) \in [0,K] \times [0,K]\} \, ds
&\geq& \frac{1}{t}\int_0^t \ind\{(\bar X_s, \bar Y_s) \in [0,K] \times [0,K]\} \, ds \\
&>& 1-\epsilon~\text{ for all}~ t > 0,
\end{eqnarray*}
and hence the normalized occupation measures of $(X,Y) $ are also tight
on $\R_+^2$.
By Prohorov's theorem \cite[Theorem 16.3]{K02},
there exists a
random probability measure $\nu$ on $\R_+^2$
and a
(possibly random) sequence
$(t_n)\subset\R_{++}$ such that $t_n\rightarrow \infty$ for which
\begin{equation}\label{e:occupation_convergence}
\Pi_{t_n} \Longrightarrow \nu
\end{equation}
as $n\rightarrow \infty$ almost surely, where $\Longrightarrow$
denotes weak convergence of probability measures on $\R_+^2$. That is,
with probability one for all bounded and continuous function
$u: \R_+^2 \to \R$
we have
\[
\int_{\R_+} u(x,y) \, \Pi_{t_n}(dx,dy)
\rightarrow
\int_{\R_+} u(x,y) \, \nu(dx,dy)
\]
as $n\rightarrow \infty$.

\begin{proposition}
The probability measure $\nu$ is almost surely a stationary distribution for $(X,Y)$ thought of as a process with state space $\R_+^2$.
\end{proposition}
\begin{proof}
Let $(P_t)_{t\geq 0}$ be the semigroup of the process $(X,Y)$ thought of
as a process on $\R_+^2$. For simplicity let us write $Z_t:=(X_t,Y_t)$ for all $t\geq 0$ and $\nu_n:=\Pi_{t_n}$.

By the Strong Law of Large Numbers for martingales, we have that for all $r\in\R_+$ and all bounded measurable functions $f$
\begin{equation*}
\lim_{k\rightarrow \infty} \frac{1}{k} \sum_{i=0}^{k-1} [f(Z_{r+(i+1)t}) - P_tf(Z_{r+it})] =0 \quad \text{almost surely}.
\end{equation*}
As a result,
\begin{eqnarray*}
\frac{1}{k}\left(\int_t^{kt}f(Z_s)\,ds - \int_0^{(k-1)t}P_t f(Z_s)\,ds \right)&=& \frac{1}{k}\sum_{i=0}^{k-1} \int_0^t [f(Z_{r+(i+1)t}) -P_tf(Z_{r+it})]\,dr\\
&\rightarrow& 0 ~\text{as}~k\rightarrow \infty \quad \text{almost surely}.
\end{eqnarray*}
This implies that
\begin{eqnarray*}
\lim_{u\rightarrow \infty}\frac{1}{u} \int_0^u [f(Z_{s+t})- P_t(Z_s)]\,ds = 0 \quad \text{almost surely}.
\end{eqnarray*}
Thus,
\begin{eqnarray}\label{e:generator}
\int f\,d\nu - \int P_t f\,d\nu &=&\lim_{n\rightarrow \infty}\left(\int f\,d\nu_n - \int P_t f\,d\nu_n\right)\nonumber\\
&=&  \lim_{n\rightarrow \infty}\frac{1}{t_n} \left[\int_0^{t_n} (f(Z_s)-P_tf(Z_s))\,ds\right]\nonumber\\
&=&   \lim_{n\rightarrow \infty}\frac{1}{t_n} \left[\int_0^{t_n-t} (f(Z_{s+t})-P_tf(Z_s))\,ds  +\int_0^t f(Z_s)\,ds - \int_{t_n-t}^{t_n} P_t f(Z_s)\,ds  \right]\nonumber\\
&=&\lim_{n\rightarrow \infty}  \frac{1}{t_n} \left[\int_0^{t_n-t} (f(Z_{s+t})-P_tf(Z_s))\,ds \right]\nonumber\\
&= &0  \quad \text{almost surely}.
\end{eqnarray}
The last result is equivalent to saying that $\nu$ is almost surely a stationary distribution for $(X,Y)$.

\end{proof}
\begin{proposition}\label{p:stationary_R_++}
There exists a stationary distribution $\pi$ of $(X,Y)$ that assigns
all of its mass to $\R_{++}^2$.
\end{proposition}
\begin{proof}
We argue by contradiction. Because the process stays in one of the
four sets $\R_{++}$, $\R_{++} \times \{0\}$, $\{0\} \times  \R_{++}$,
$\{(0,0)\}$ when it is started in the set, any stationary distribution
for $(X,Y)$ thought of as a process on $\R_+^2$
can be written as a convex combination of stationary distributions that
respectively assign all of their masses to one of the four sets, should such
a stationary distribution exist  for the given set.
Suppose there is no stationary distribution that
is concentrated on $\R_{++}^2$. Then, any stationary distribution  is the convex combination of stationary distributions that respectively assign
all of their mass to the three sets
$\R_{++} \times \{0\}$, $\{0\} \times  \R_{++}$,
and $\{(0,0)\}$, and hence any stationary distribution is
of the form
\begin{equation*}
p_X \mu_X + p_Y \mu_Y + p_0 \delta_{(0,0)},
\end{equation*}
where the random variables $p_X, p_Y, p_0$ are nonnegative and
$p_X+p_Y + p_0 =1$ almost surely, and $\mu_X = \rho_{\bar X} \otimes \delta_0$ and  $\mu_Y = \delta_0 \otimes \rho_{\bar Y}$ for $\rho_{\bar X}$ and $\rho_{\bar Y}$ the unique stationary distributions of $\bar X$ and $\bar Y$. Next, we proceed as in Proposition \ref{P:M extinction} to find the limit of $\frac{\log X_{t_n}}{t_n}$.
Let us first argue that
\begin{eqnarray}\label{e:ergodic}
\lim_{n\rightarrow \infty}\frac{1}{t_n} \int_0^{t_n} X_s\,ds &=& \int_{\R_{+}^2} x \, \nu(dx,dy)\\
\lim_{n\rightarrow \infty}\frac{1}{t_n} \int_0^{t_n} Y_s\,ds &=& \int_{\R_{+}^2} y \, \nu(dx,dy) \quad \text{almost surely}.\nonumber
\end{eqnarray}

Note that the infinitesimal generator of $(\log X, \log Y)$ thought
of as a process on $\R^2$ is uniformly
elliptic with smooth coefficients and so it has smooth transition densities
(see, for example, Section 3.3.4 of \cite{MR2410225}). Moreover,
an application of a suitable minimum principle for the Kolmogorov
forward equation (see, for example, Theorem 5 in Section 2 of Chapter 2
of \cite{MR0181836}) shows that the transition densities are
everywhere strictly positive.  It follows that $(X,Y)$
thought of as a process on $\R_{+}^2$ has smooth transition densities
that are everywhere positive.

Because the process
$\bar X$ also has smooth, every positive transition densities for
similar reasons, the almost sure behavior of the $\bar X$ started from a fixed point is the same as it is starting from its stationary distribution $\rho_{\bar X}$. As a result, we get by Birkhoff's pointwise
ergodic theorem \cite[Theorem~10.6]{K02} that, for all $K>0$,
\begin{equation*}
 \lim_{n \to \infty} \frac{1}{t_n} \int_0^{t_n} \bar X_s \ind\{\bar X_s > K\} \, ds = \E^{\rho_{\bar X}} [\bar X_s \ind\{\bar X_s > K\}]
\end{equation*}
$\Pr^{x}$ almost surely for any $x\in\R_+$. Therefore, by dominated convergence
\begin{eqnarray*}
\lim_{K \to \infty} \lim_{n \to \infty} \frac{1}{t_n} \int_0^{t_n} \bar X_s \ind\{\bar X_s > K\} \, ds  = \lim_{m \to \infty} \E_{\rho_{\bar X}} [\bar X_s \ind\{\bar X_s > K\}]= 0.
\end{eqnarray*}

The following inequalities are immediate due to the positivity of the terms
\begin{equation}
\label{e:ineq}
\begin{split}
\frac{1}{t_n} \int_0^{t_n} X_s \ind \{ X_s \leq K\} \,ds
&\leq \frac{1}{t_n} \int_0^{t_n} X_s\, ds\\
&= \frac{1}{t_n} \int_0^{t_n} X_s \ind \{ X_s \leq K\} \,ds \\
& \quad + \frac{1}{t_n}  \int_0^{t_n} X_s \ind \{ X_s > K\}\, ds. \\
\end{split}
\end{equation}
Recall that $X_t\leq \bar X_t$ for all $t\geq 0$
and hence
\begin{equation*}
\frac{1}{t_n} \int_0^{t_n} X_s \ind \{ X_s > K\} \,ds \leq \frac{1}{t_n} \int_0^{t_n} \bar X_s \ind \{ \bar X_s > K\} \,ds.
\end{equation*}
This implies
\begin{equation*}
\limsup_{n\rightarrow\infty} \frac{1}{t_n} \int_0^{t_n} X_s \ind \{ X_s > K\} \,ds \leq \limsup_{n\rightarrow\infty} \frac{1}{t_n} \int_0^{t_n} \bar X_s \ind \{ \bar X_s > K\}\, ds,
\end{equation*}
and therefore
\begin{equation}
\label{e:limsup}
\begin{split}
0 & \leq \lim_{K\rightarrow \infty} \limsup_{n\rightarrow\infty} \frac{1}{t_n} \int_0^{t_n} X_s \ind \{ X_s > K\} \,ds \\
& \leq \lim_{K\rightarrow \infty} \limsup_{n\rightarrow\infty} \frac{1}{t_n} \int_0^{t_n} \bar X_s \ind \{ \bar X_s > K\} \,ds =0. \\
\end{split}
\end{equation}
By \eqref{e:occupation_convergence}
and Theorem 4.27 of \cite{K02},
\begin{equation*}
\lim_{n\rightarrow\infty} \frac{1}{t_n} \int_0^{t_n} X_s \ind \{ X_s \leq K\}\, ds = \int_{\R_{++}^2} x \ind\{x\leq K\}\, \nu(dx,dy).
\end{equation*}
for any $K$ such that
\[
\nu(\{K\}\times \R_+)=0.
\]
While this last condition need not hold a priori
for all $K$, we can only have
\[
\nu(\{K\}\times \R_+)>0
\]
for countably many $K$, so there exists a sequence $(K_m)\subset \R_+$ such that $K_m\rightarrow \infty$ as $m\rightarrow\infty$ with
\[
\nu(\{K_m\}\times \R_+)=0.
\]

By dominated convergence,
\begin{equation}\label{e:dominated}
\begin{split}
\lim_{m\rightarrow \infty } \lim_{n\rightarrow\infty} \frac{1}{t_n} \int_0^{t_n} X_s \ind \{ X_s \leq K_m\}\, ds
& = \lim_{K\rightarrow \infty } \int_{\R_{+}^2} x \ind\{x\leq K\} \,\nu(dx,dy) \\
& = \int_{\R_{+}^2} x \,\nu(dx,dy). \\
\end{split}
\end{equation}
Combining \eqref{e:ineq}, \eqref{e:limsup} and \eqref{e:dominated} gives \eqref{e:ergodic}.

It follows from It\^o's formula,
the observation $\I(\alpha,\alpha)=0$, \eqref{e:ergodic},
and the fact that $\lim_{n\rightarrow\infty}\frac{U_{t_n}}{t_n}=0$
that
\begin{eqnarray*}
\lim_{n\rightarrow \infty} \frac{\log X_{t_n}}{t_n} &=& \mu\cdot\alpha - \frac{\sigma_X^2}{2} - \E^\nu [aX_t + bY_t]\\
&=& p_X\left(\mu\cdot\alpha -a\E^{\bar \rho_X}[X_t] - \frac{\sigma_X^2}{2}\right)\\
&~&+~p_Y\left(\mu\cdot\alpha -b\E^{\bar \rho_Y}[Y_t]-\frac{\sigma_X^2}{2}\right) + p_0 \left(\mu\cdot\alpha -\frac{\sigma_X^2}{2}\right)\\
&=& p_X \I(\alpha,\alpha) + p_Y \I(\alpha,\beta) + p_0\left(\mu\cdot\alpha -\frac{\sigma_X^2}{2}\right)\\
&=& p_Y \I(\alpha,\beta) + p_0\left(\mu\cdot\alpha -\frac{\sigma_X^2}{2}\right)
\quad \text{almost surely}.
\end{eqnarray*}

By assumption, $\I(\alpha,\beta)>0$ and
we have already observed  that $\mu\cdot\alpha -\frac{\sigma_X^2}{2}>0$. Because  $\bar X_t$ converges in distribution as $t \to \infty$
to a distribution that assigns all of its mass to $\R_{++}^2$, it follows that
$\frac{\log \bar X_{t_n}}{t_n}$ converges in probability to $0$.
However, since
$X_t \le \bar X_t$ for all $t \ge 0$ it follows that
$p_Y \I(\alpha,\beta) + p_0\left(\mu\cdot\alpha -\frac{\sigma_X^2}{2}\right)
\le 0$ and hence
\begin{equation}
p_Y=p_0=0 \quad \text{almost surely}.
\end{equation}
The same argument applied to $(Y_t)_{t\geq 0}$ establishes
\begin{equation}
p_X=p_0=0 \quad \text{almost surely}.
\end{equation}
Therefore, $p_X=p_Y=p_0=0$, and this
contradicts the assumption that $p_X+p_Y+p_0=1$.
\end{proof}

We can now finish the proof of Theorem \ref{T:stationary_positive_Lyapunov}.
\begin{proof}
Proposition \ref{p:stationary_R_++} implies that $(X,Y)$ has a stationary distribution $\pi$ on $\R_{++}^2$.
By Theorem 20.17 from \cite{K02}, our process $(X,Y)$ is either Harris recurrent or uniformly transient. We say that
 $(X_t,Y_t)\rightarrow \infty$ almost surely as $t\rightarrow \infty$ if  $\ind_K(X_t,Y_t)\rightarrow 0$ as $t \to \infty$
 for any compact set $K\subset \R_{++}^2$.
Theorem 20.21 from \cite{K02} gives that if $(X,Y)$ is transient, then $(X_t,Y_t)\rightarrow \infty$ and so  $(X,Y)$ cannot have a stationary distribution. Hence, since we know our process has a stationary distribution $\pi$, it must be Harris recurrent.  Theorem 20.21 from \cite{K02} then gives us equation \eqref{e_pathwise_limit}.

Theorem 20.18 from \cite{K02}, 20.18 gives that any Harris recurrent Feller process on $\R_{++}^2$ with strictly positive transition densities
has a locally finite invariant measure that is equivalent to Lebesgue measure and is unique up to a normalization. We already know that we have a stationary distribution, so this distribution is unique and has an almost everywhere strictly positive density with respect to Lebesgue measure. Theorem 20.12 from \cite{K02} says that any Harris recurrent Feller process is strongly ergodic, and so equation \eqref{e_ergodic} holds.

\end{proof}
\begin{remark}
In Theorem 3.1 of \cite{ZC13}, the authors claim to show that the system of SDE describing $(X,Y)$ always has a unique stationary distribution. We note that their use of moments just checks tightness in $\R_{+}^2:=[0,\infty)^2$ and not in $\R_{++}^2 = (0,\infty)^2$. It does not stop mass going off to
$\R_{+}^2 \setminus \R_{++}^2 = (\R_+ \times \{0\}) \cup (\{0\} \times \R_+)$, which is exactly what can happen in our case. Thus, their proof only shows the existence of a stationary distribution on $\R_{+}^2$ -- it does not show the existence of a stationary distribution on $\R_{++}^2$. Furthermore, their proof for the uniqueness of a stationary distribution on $\R_{+}^2$ breaks down because their assumption of irreducibility is false. The process $(X,Y)$ is irreducible on $\R_{++}^2$, but it is not irreducible on $\R_{+}^2$ since $P_t((0,0),U):=\Pr^{(0,0)}\{(X_t,Y_t)\in U\}=0$ for any open subset $U$ that lies in the interior of $\R_{+}^2$.
If we work on $\R_{+}^2$, it is not true that the diffusion $(X,Y)$ has a unique stationary distribution. We can obtain infinitely many stationary distributions on $\R_{+}^2$ of the form $(u \rho_{\bar X} + v \delta_0)\otimes \delta_0$ where $\rho_{\bar X}$ is the unique stationary distribution of $\bar X$ on $\R_{++}$ and $u,v\in \R_+$ satisfy $u+v=1$.
\end{remark}

\section*{Appendix D: Proof of Theorem~\ref{T:ESS}}

Assume that the matrix $\Sigma$ is positive definite and
that the dispersion proportion vector
$\alpha$ is such that $\mu\cdot \alpha- \alpha\cdot \Sigma \alpha/2>0$
so that a population playing the strategy $\alpha$ persists. Under these
assumptions the function $\beta \mapsto \I (\alpha,\beta)$ is strictly concave. Hence, by the method of Lagrange multipliers, $\I(\alpha,\beta)<0$ for all $\beta\neq \alpha$ and $\alpha_i>0$ for all $i$ if and only if there exists a constant, which we denote by $\lambda$, such that
\begin{equation}\label{eq:lagrange}
\left.\lambda = \frac{\partial \I}{\partial \beta_i } (\alpha,\beta)\right\vert_{\beta=\alpha} = \mu_i - \kappa_i \alpha_i (\mu\cdot \alpha - \alpha\cdot \Sigma \alpha/2)/\la \alpha,\alpha\ra_\kappa  - \sum_j \alpha_j \sigma_{ij}
\end{equation}
for all $i$. Multiplying \eqref{eq:lagrange} by $\alpha_i$ and summing with respect to $i$, we get
\[
\begin{aligned}
\lambda &= \mu \cdot \alpha - \la \alpha,\alpha\ra_\kappa (\mu\cdot \alpha - \alpha\cdot \Sigma \alpha/2)/\la \alpha,\alpha\ra_\kappa  - \alpha \cdot \Sigma \alpha\\
&=- \alpha \cdot \Sigma \alpha/ 2
\end{aligned}
\]
This expression for the Lagrange multiplier and \eqref{eq:lagrange} provide the characterization of a mixed ESS in equation \eqref{E:ESS} when $\alpha_i>0$ for all $i$. The characterization of the more general case of $\alpha_i>0$ for at least two patches follows similarly by restricting the method of Lagrange multiples to the appropriate face of the
probability simplex.

Suppose that $\mu_i -\sigma_{ii}/2>0$ so that a population remaining in patch $i$ and not dispersing to other patches persists. The strategy $\alpha_i=1$ and $\alpha_j=0$ for all $j\neq i$ is an ESS only if
\[
\left.\frac{\partial \I}{\partial \beta_j }(\alpha,\beta)\right\vert_{\beta=\alpha}- \left.\frac{\partial \I }{\partial \beta_i}(\alpha,\beta)\right\vert_{\beta=\alpha}<0
\]
for all $j\neq i$. Evaluating these partial derivatives gives the criterion  \eqref{E:ESS2} for the pure ESS.

We conclude by considering the case $n=2$. Define the function $g:[0,1]\to \mathbb{R}$ by
\begin{equation*}
\begin{split}
g(a)&= \left.\frac{\partial \I}{\partial \beta_1}((a_1,a_2),(b_1,b_2))\right\vert_{(a_1,a_2)=(a,1-a), (b_1,b_2)=(a,1-a)}\\
&~-\left.\frac{\partial \I}{\partial \beta_2}((a_1,a_2),(b_1,b_2))\right\vert_{(a_1,a_2)=(a,1-a), (b_1,b_2)=(a,1-a)}.\\
\end{split}
\end{equation*}
The inequalities \eqref{E:ESS2} for the pure strategies $(1,0)$ and $(0,1)$, respectively, correspond to $g(0)<0$ and $g(1)>0$, respectively. Hence, when these inequalities are reversed, the intermediate value theorem implies there exists $a\in (0,1)$ such that $g(a)=0$. Such an $a$ satisfies the mixed ESS criterion \eqref{E:ESS} and, therefore, is an ESS.

\bigskip
\noindent
{\bf Acknowledgments.}  The authors thank Dan Crisan, Alison Etheridge, Tom Kurtz, and Gregory Roth for helpful discussions.

\end{document}